\documentclass[10pt]{elsarticle}
\usepackage{amssymb}
\usepackage{amsthm}
\usepackage{amsmath}
\usepackage{graphics}
\usepackage{epsfig}
\usepackage{epstopdf}
\usepackage{caption}
\usepackage{commath}
\usepackage{subcaption}
\usepackage{float}
\usepackage{xcolor}
\usepackage{soul,verbatim}
\usepackage[cm]{fullpage}

\newtheorem{prop}{Proposition}
\newtheorem{example}{Example}

\newcommand{\qp}[1]{\ensuremath{\!\left({#1}\right)}}

\newcommand{\qc}[1]{\ensuremath{\!\left\{{#1}\right\}}}

\newcommand{\beatrice}[1]{{\color{olive} B: #1 }}

\begin{document}
\begin{frontmatter}
\title{A numerical implementation of the unified Fokas transform for evolution problems on a finite interval}

\author[a]{Emine Kesici}
\address[a]{Department of Mathematics, Istanbul Technical University, Maslak 34469,
         Istanbul,  Turkey.}
\author[b]{Beatrice Pelloni}
\address[b]{Department of Mathematics, Heriot-Watt University, Edinburgh EH14 4AS, UK.}
\author[c]{Tristan Pryer}
\address[c]{
  Department of Mathematics and Statistics,
  University of Reading,
  Whiteknights,
  PO Box 220,
  Reading RG6 6AX,
  UK.
}
\author[d]{David Smith}
\address[d]{
	Division of Science,
	Yale-NUS College,
	16 College Avenue West,
	\#01-220 138527,
	Singapore.
}

\begin{abstract}
  We present the numerical solution of two-point boundary value problems for a third order linear PDE, representing a linear evolution in one space dimension.  The difficulty of this problem is in the numerical imposition of the boundary conditions, and to our knowledge, no such computations exist. Instead of computing the evolution numerically, we evaluate the solution representation formula obtained by the unified transform, also known as Fokas transform.  This representation involves complex line integrals, but in order to evaluate these integrals numerically, it is necessary to deform the integration contours using appropriate deformation mappings. We formulate a strategy to implement effectively this deformation, which allows us to obtain accurate numerical results.
\end{abstract}
\end{frontmatter}

\section{Introduction}

In a series of papers over the last ten years, two of the authors, in collaboration with Fokas, have conducted an extensive analysis of boundary value  problems on a finite interval for linear evolution partial differential equations (PDEs) in one space variable \cite{FP2001a,FP2005a,Pel2004a,Pel2005a, Smi2011a}.  This analysis uncovered results  that are somewhat surprising, given the fact that the problem is a linear problem in one spatial variable. The novel ingredient that allowed a broader understanding of the structure of such boundary value problems is the approach known as the unified transform,  or Fokas transform. This approach, pioneered by Fokas and significantly extended during the past 15 years by a number of authors, gives a unified way to treat boundary value problems for linear and integrable nonlinear PDEs in two independent variables (for a general account and bibliography, see \cite{Fok2008a,fokas2014unified}). In its general form,  the transform and its inverse are rigorously obtained through the solution of a so-called Riemann-Hilbert problem, a classical problem in complex analysis \cite{AF1997a, itsAMS}. Indeed, this transform is at its heart  a complex variable approach. The shift in perspective from the use of the classical Fourier analysis approach, a real variable transform, to using a transform in the complex plane enabled a broader and deeper understanding of the structure of these boundary value problems.  Using the unified transform, it has been possible  to classify the particular boundary value problems, for  {\em linear} evolution PDEs in two variables, for which a series representation of the solution of the problem {\em does not exist} - problems that behave very differently from a 2-point or periodic boundary value problem for the prototypical evolution PDEs in one space dimension, such as the heat equation  \cite{PS2014a,FS2016a}.

This approach yields an explicit representation of the solution of the boundary value problem in the form of a complex contour integral.  This representation is more general than the classical series representation, and can be shown to be equivalent to it whenever such series representation exists - for example when the boundary conditions are periodic.

In a separate development, Olver \cite{Olv2010a} discovered,  by careful numerical investigation then confirmed by rigorous computations, that the solution of a periodic boundary value problem for linear evolution PDEs can display a phenomenon that he called {\em dispersive quantisation}. It is natural then to investigate whether this quantisation is only supported in a periodic setting, and in general how boundary conditions qualitatively affect the solution at all or at specific times.

We set therefore to try and investigate whether dispersive quantisation could occur for different boundary conditions. In particular, our aim was to investigate, numerically in the first instance, whether there is a qualitative difference between the solutions of the same PDE posed with boundary conditions that support the existence of a series solution representation, or with boundary conditions for which such a series solution representation does not exist.

To this end,  we consider boundary value problems for a specific third-order linear evolution PDE, sometimes known as  Airy's equation, defined on the finite interval $[0,L]$.
Namely, we consider the following class of initial, boundary value problems:
\begin{equation}\label{lkdv}
\begin{split}
  q_{t}(x,t)+q_{xxx}(x,t)&=0, \quad x\in[0,L], \quad t>0 \\
  q(x,0)&=q_{0}(x), \quad x\in[0,L]\\
  q(0,t)&=f_{0}(t), \quad t>0\\
  q(L,t)&=g_{0}(t), \quad t>0\\
  q_{x}(L,t)&=\alpha \, q_{x}(0,t), \quad t>0, \alpha\in\mathbb R,
\end{split}
\end{equation}
{where $q_0$, $f_0$ and $g_0$ are prescribed functions, compatible at the point $(x=0,L,t=0)$, the corners of the domain.}

Equation (\ref{lkdv}), or rather its close sibling $q_t+q_x+q_{xxx}=0,$ is the linearisation of the famous Korteweg-de Vries equation which models shallow water waves. Hence, this equation has significant importance from the point of view of applications. However, our motivation in studying this particular problem is that it is the simplest possible problem for which the phenomenon referred to above, namely the lack of a series representation for the solution, can manifest itself given specific boundary conditions.

The behaviour of the solution to \eqref{lkdv} depends essentially on the coupling constant $\alpha$. This is easily justified by a formal integration by parts. Indeed, to illustrate the effect of different values of $\alpha$ consider
\begin{eqnarray}\label{cons-quan}
\begin{split}
\frac{d}{dt} \|q\|_{2}^{2} &:= \frac{d}{dt} \int_{0}^{L} |q|^2 dx =-2 \int_{0}^{L} q \, q_{xxx} \, dx = -2 (q \, q_{xx})\bigg|_{x=0}^{x=L}+2\int_{0}^{L} q_{xx} \, q_{x} \, dx =(-2 q \, q_{xx} + q_{x}^2) \bigg|_{x=0}^{x=L} \\
                           &= 2 q(0,t)\,q_{xx}(0,t)- 2q(L,t)\,q_{xx}(L,t) - q_{x}^{2}(0,t) + q_{x}^{2}(L,t).
\end{split}
\end{eqnarray}
Imposing the boundary condition $q_{x}(L,t) = \alpha \, q_{x}(0,t)$, \eqref{cons-quan} simplifies to give:
\begin{equation*}
\frac{d}{dt} \|q\|_{2}^{2} = 2 q(0,t)\,q_{xx}(0,t)- 2q(L,t)\,q_{xx}(L,t) + (\alpha^2 -1) \, q_{x}^{2}(0,t).
\end{equation*}
When $f_0(t) \equiv 0$ and $g_0(t) \equiv 0$, then the energy is conserved in time if $|\alpha|=1$ and energy decreases in time when $|\alpha|<1$, that is the equation becomes dispersive.

In this paper, we study numerically the effect of varying $\alpha$ values. Our starting point is the integral representation of the solution given by the Fokas transform, and its numerical evaluation. A similar numerical approach was pioneered in \cite{flyerfokas2008}. In this paper, Fokas and Flyer evaluated the contour integral representation for the solution of specific boundary value problems for the linear KdV equation, but {\em only posed  on the half line}, i.e. in the limit as $L\to \infty$. In this case, the integrand is a function which is analytic with respect to the complex variable of integration.  It turns out that the problem of selecting a contour in an optimal way for the numerical evaluation of the relevant integral is a very delicate issue.  This is due to the fact that, unlike the case when the problem is posed on a half line, in the case of a problem posed on a finite interval the integrand is no more analytic, but rather meromorphic, and the evaluation of the residues at the poles may be the dominant contribution to the computation.

In this paper, we concentrate on devising a strategy for this numerical evaluation. This strategy will then be put to use to investigate the original quantisation question in subsequent work.

The organisation of the paper is as follows: In section \ref{section2}, we summarise the unified transform method, and discuss the integral representation of the solution of problem \ref{lkdv}. In section \ref{section3}, the numerical strategy for the finite interval case will be given by providing comparison with half-line case. In section \ref{section4}, some numerical examples on finite interval with the different coupling constant $\alpha \in [0,1]$ will be presented. In conclusion, we discuss the numerical results and indicate future directions of investigation.
\section{The integral representation of the solution and the global relation}\label{section2}
The unified transform is a general methodology for studying boundary value problems for linear and integrable nonlinear PDEs in two independent variables. The particular case of the equation (\ref{lkdv}) considered in the present paper has been studied  in detail in \cite{Pel2005a, Smi2011a}.
We  refer to these works for more details, and limit ourselves to a brief summary containing the main ideas of the method.

The starting point of the analysis is an alternative equivalent formulation of the PDE. In the nonlinear integrable case, that was the original motivation for this development, this  is know as a Lax pair formulation. In the linear evolution case, it is straightforward to verify that any PDE of the general form
\begin{equation}
\partial_tq+w(-i\partial_x)q=0,\quad
\label{genpde}
\end{equation}
 is equivalent to  the differential formulation
\begin{equation}\label{Greenfunc}
(e^{-ikx+w(k)t} \,q)_{t} -(e^{ikx+w(k)t} \, X)_{x} =0,\qquad \forall k\in\mathbb C.
\end{equation}
From this starting point, one can deduce two consequences:
\begin{itemize}
\item
A constraint involving certain transforms of all initial and boundary values in terms of a spectral parameter, $k$, is valid for $k\in\mathbb C$.  This constraint is usually called the {\em global relation}, and it is the key ingredient of the method.
\item
A complex contour  integral representation for the function $q(x,t)$ in terms of all initial and boundary values. This obtained by a formal Fourier inversion  and a contour deformation in the $k$ complex plane.

\end{itemize}
Rather then prove this methodology in general, we describe in brief detail how it is implemented for the particular case under consideration, namely equation (\ref{lkdv}), which corresponds to
$$
w(k)=-ik^3,\qquad X(x,t,k)=k^2 q -ikq_{x}-q_{xx},
$$
in formulations (\ref{genpde}) and (\ref{Greenfunc}). We stress that the same approach works for any PDE of the form (\ref{genpde}).

We apply Green's theorem to the differential form \eqref{Greenfunc} in the convex domain $[0,L]\times [0,t]$ to obtain the global relation
\begin{equation}
e^{w(k)t}\int_{0}^{L}e^{-ikx}q(x,t)dx =\int_{0}^{t} e^{-ikL+w(k)s}X(L,s,k)ds -\int_{0}^{t}e^{w(k)s}X(0,s,k)ds +\int_{0}^{L}e^{-ikx}q_{0}(x)dx.
\label{gengr}
\end{equation}
Define the following {\em spectral functions}, which are functions of the complex variable $k$:
\begin{eqnarray*}
  \label{eq:qhat}
  \widehat{q}(k,t)&=&
  \int_{0}^{L}e^{-ikx}q(x,t)dx, \qquad
  \widehat{q_{0}}(k)
  =\int_{0}^{L}e^{-ikx}q_{0}(x)dx\\
\widetilde{f}(k,t)&=&
\int_{0}^{t}e^{w(k)s}X(0,s,k) ds=:k^2 \widetilde{f_{0}}(k,t)-i k \widetilde{f_{1}}(k,t)-\widetilde{f_{2}}(k,t),\\
\widetilde{g}(k,t)&=&
\int_{0}^{t}e^{-ikL+w(k)s}X(L,s,k)ds=:k^2 \widetilde{g_{0}}(k,t)-i k \widetilde{g_{1}}(k,t)-\widetilde{g_{2}}(k,t),
\end{eqnarray*}
where we have defined $\widetilde{f_{0}},\widetilde{f_{1}},\widetilde{f_{2}}$ and $\widetilde{g_{0}},\widetilde{g_{1}},\widetilde{g_{2}}$ through the definition of $X$ and powers of $k$. The global relation (\ref{gengr}) can then be written as
\begin{equation}\label{globalrel}
\widetilde{f}(k,t)-e^{-ikL}\widetilde{g}(k,t)=\widehat{q_{0}}(k)-e^{w(k)t}\widehat{q}(k,t).
\end{equation}
The solution representation is found by applying inverse Fourier transforms to \eqref{globalrel}, and deforming contours  - a procedure rigorously justified  in this case by an application of Jordan's lemma. We then obtain
\begin{equation}\label{integralsoln}
q(x,t)=\frac{1}{2\pi}\qc{\int_{\Re} e^{ikx+ik^3t} \, \widehat{q_{0}}(k) \, dk - \int_{\partial D^{+}} e^{ikx+ik^3t} \, \widetilde{f}(k,t) \, dk  - \int_{\partial D^{-}} e^{ik(x-L)+ik^3t} \, \widetilde{g}(k,t) \, dk},
\end{equation}
where $D^{\pm}=\{k \in \mathbb{C}^{\pm}: \Re(w(k))\leq 0\}$ and the orientation of the integration path is such that the interior of the domain remains on the left, see figure \ref{fig:fig1} for an illustration.
\begin{figure}[H]
\centering
\caption{\label{fig:fig1} {The regions $D^{\pm}=\{k \in \mathbb{C}^{\pm}: \Re(w(k))\leq 0\}$ where $D^{-}=D_{1}^{-}\cup D_{2}^{-}$.}}
\includegraphics{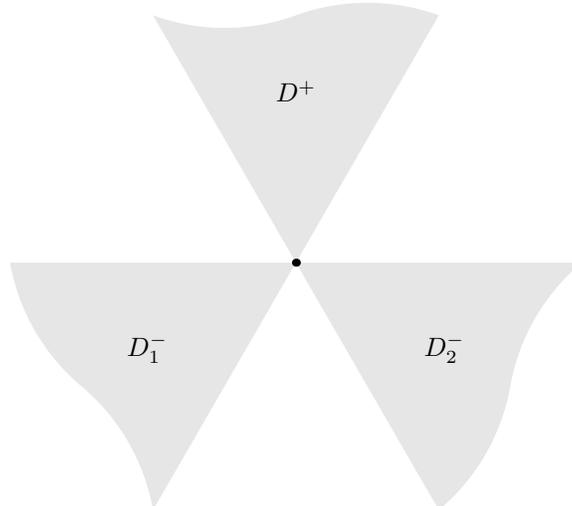}
\end{figure}

At this point we have a complex contour representation of the solution $q(x,t)$ that still depends on all the boundary values, both known and unknown. Hence this representation is not yet explicit. The crucial step to arrive to an explicit representation in terms only of the prescribed data of the problem lies in the analysis of the global relation. Namely, we seek an expression for the integrand in (\ref{integralsoln}) depending only the prescribed initial and  boundary conditions.

\subsection{Analysis of the global relation} \label{ssec:AnalysisGR}
Observe that  the spectral functions $\widehat{q_{0}}$, $\widetilde{f_{0}}$, $\widetilde{g_{0}}$ and $\widetilde{g_{1}}=\alpha \widetilde{f_{1}}$ are obtained as transforms of the prescribed data,  and are therefore known. However the three functions  $\widetilde{f_{1}}$, $\widetilde{f_{2}}$ and $\widetilde{g_{2}}$ cannot be computed from the given data.

Since we have three unknown functions, and only one equation, we seek additional algebraic equations.
The key is to exploit the invariance properties of the spectral functions in the complex $k$ plane. Observe that the functions $\widetilde f_i$, $i=1,2,3$ are functions of $k$ only through $w(k)=-ik^3$. Since $w(k)$ is invariant under rotation by $2\pi/3$, the substitutions $k\rightarrow \tau k$ and $k\rightarrow \tau^2 k$ leave these functions invariant for $\tau=\exp(2\pi i/3)$.
Therefore, evaluating the global relation at $\tau k$ and $\tau^2 k$, we find a system of three equations with three unknowns:
\begin{eqnarray}\label{system}
(k^2\widetilde{f_{0}}-ik\widetilde{f_{1}}-\widetilde{f_{2}})-e^{-ikL}(k^2\widetilde{g_{0}}-ik\widetilde{g_{1}}-\widetilde{g_{2}})&=&\widehat{q_{0}}(k)-e^{-ik^3t}\,\widehat{q}(k,t) \nonumber \\
(\tau^2 k^2\widetilde{f_{0}}-i\tau k\widetilde{f_{1}}-\widetilde{f_{2}})-e^{-i\tau kL}(\tau^2 k^2\widetilde{g_{0}}-i\tau k\widetilde{g_{1}}-\widetilde{g_{2}})&=&\widehat{q_{0}}(\tau k)-e^{-ik^3t}\,\widehat{q}(\tau k,t)\\
(\tau^4 k^2\widetilde{f_{0}}-i\tau^2 k\widetilde{f_{1}}-\widetilde{f_{2}})-e^{-i\tau^2 kL}(\tau^4 k^2\widetilde{g_{0}}-i\tau^2 k\widetilde{g_{1}}-\widetilde{g_{2}})&=&\widehat{q_{0}}(\tau^2 k)-e^{-ik^3t}\,\widehat{q}(\tau^2 k,t). \nonumber
\end{eqnarray}



After solving the system \eqref{system} and substituting the result into the expressions for $\widetilde{f}(k,t)$ and $\widetilde{g}(k,t)$ we obtain
\begin{align*}
\widetilde{f}(k,t) &= k^{2}\widetilde{f_{0}}+\frac{1}{\Delta(k)} \, \qc{N(k,t) A_{1}(k)-\tau^2 N(\tau k,t) A_{2}(k)-N(\tau^2 k,t)A_{3}(k)}, \\
\widetilde{g}(k,t) &= k^{2}\widetilde{g_{0}}+\frac{1}{\Delta(k)} \, \qc{N(k,t) B_{1}(k)-\tau^2 N(\tau k,t) B_{2}(k)-N(\tau^2 k,t)B_{3}(k)},
\end{align*}
where
\begin{align} \notag
N(k,t) &= N_{\mathrm{data}}(k) - e^{-ik^3t}\widehat{q}(k,t), \qquad \qquad N_{\mathrm{data}}(k) = \widehat{q_{0}}(k)-k^2 \widetilde{f_{0}} + e^{-ikL} k^2 \widetilde{g_{0}}, \\
\label{Ndata}\\
\Delta(k) &= \tau\left\{\left[e^{-ikL}+\tau e^{-i\tau kL}+\tau^2e^{-i\tau^2kL}\right]+\alpha
\left[e^{ikL}+\tau e^{i\tau kL}+\tau^2e^{i\tau^2kL}\right]\right\},
\label{NDelta}
\end{align}
and
\begin{align*}
A_{1}(k)&=\tau \alpha e^{ikL}+\tau^2 e^{-i\tau k L} + e^{-i\tau^2 k L}, &
A_{2}(k)&=e^{-ikL}-\alpha e^{i \tau kL}, &
A_{3}(k)&=e^{-ikL}-\alpha e^{i\tau^2 k L},\\
B_{1}(k)&=-\tau-\alpha e^{-i\tau kL}-\tau^2\alpha e^{-i\tau^2 k L}, &
B_{2}(k)&=1-\alpha e^{-i \tau^2 kL}, &
B_{3}(k)&=1-\alpha e^{-i\tau k L}.
\end{align*}
Note that the function $\Delta(k)$ has infinitely many zeros. The asymptotic location of these zeros depends crucially on the specific value of $\alpha$, and this will turn out to be critical for our calculations.

These expressions for $\widetilde{f}$ and $\widetilde{g}$ are now in terms of data and $\widehat{q}(k,t)$, so substituting these expressions into equation~\eqref{integralsoln} yields an integral representation of the solution depending only on data and $\widehat{q}(k,t)$:
\begin{multline} \label{eqn:SolnqWithqt}
	2\pi q(x,t)=\int_{\Re} e^{ikx+ik^3t} \, \widehat{q_{0}}(k) \, dk - \int_{\partial D^{+}} e^{ikx+ik^3t} \, \begin{bmatrix}\mbox{part of }\widetilde{f}(k,t)\\\mbox{depending on data}\end{bmatrix} \, dk
	- \int_{\partial D^{-}} e^{ik(x-L)+ik^3t} \, \begin{bmatrix}\mbox{part of }\widetilde{g}(k,t)\\\mbox{depending on data}\end{bmatrix} \, dk \\
	+ \int_{\partial D^{+}} e^{ikx} \, \frac
	{\widehat{q}(k,t) A_{1}(k)-\tau^2 \widehat{q}(\tau k,t) A_{2}(k)-\widehat{q}(\tau^2k,t)A_{3}(k)}{\Delta(k)} \, dk \\
	+ \int_{\partial D^{-}} e^{ik(x-L)} \, \frac {\widehat{q}(k,t) B_{1}(k)-\tau^2 \widehat{q}(\tau k,t) B_{2}(k)-\widehat{q}(\tau^2k,t)B_{3}(k)}{\Delta(k)} \, dk.
\end{multline}
Of course, it must be justified that splitting the integral along $\partial D^+$ (similarly the integral along $\partial D^-$) in this way yields convergent integrals.
Assuming such a justification can be provided, the aim would be to deduce that  the latter two integrals of equation~\eqref{eqn:SolnqWithqt} give a vanishing contribution - using analytic consideration, namely an asymptotic argument and Jordan's lemma. If this could be achieved, the first three integral terms of equation~\eqref{eqn:SolnqWithqt} would provide a solution representation for $q$ depending only upon the data of the problem.

In order to justify the arguments sketched in the above paragraph, there are three essential ingredients:
\begin{enumerate}
	\item[(A)]{
		The integrands in the latter four integrals of equation~\eqref{eqn:SolnqWithqt} are meromorphic functions of $k$, which have poles only at nontrivial zeros of $\Delta$.
	}
	\item[(B)]{The zero of $\Delta$ at $k=0$ corresponds to a removable singularity of each integrand. All other zeros of $\Delta$ are exterior to $D$, i.e.\ strictly to the right of $\partial D^\pm$.
	}
	\item[(C)]{
		The part of the integrand excluding the first exponential factor, in the fourth (respectively, fifth) integral of equation ~\eqref{eqn:SolnqWithqt}, decays as $k\to\infty$ from within the closure of $D^+$ (respectively, the closure of $D^-$).
	}
\end{enumerate}

\subsubsection*{The case $|\alpha|<1$}
Statement~(A) follows from the definitions of $N$, $\Delta$, $A_j$, $B_j$ above. Applying the methods of Langer~\cite{Lan1931a}, it can be shown that statement~(B) holds if $|\alpha|<1$, see \cite{Pel2005a, Smi2012a}.
Finally, still assuming $|\alpha|<1$, statement~(C) can be established using a geometric argument on the relative growth rates of $e^{i\tau^jk}$ for $j=0,1,2$, together with an integration by parts style asymptotic argument. Now~(A) and~(B) guarantee that there are no poles of the integrands lying on or to the left of the contours for the latter two integrals of equation~\eqref{eqn:SolnqWithqt}, so~(C) and Jordan's lemma imply that both integrands evaluate to $0$.

\subsubsection*{The case $|\alpha|=1$}
In this case, statement~(B), and therefore statement ~(C), are false.
Indeed it can be shown that, for generic data, each nontrivial zero of $\Delta$ is a pole of the integrands of the latter four integrals of~\eqref{eqn:SolnqWithqt}. This means that the latter four integrals of~\eqref{eqn:SolnqWithqt} do not converge. However, following a method suggested by~\cite{FP2001a} and implemented in~\cite{Smi2012a} (see also the review article \cite{deconinckSIAM}), we modify the contours $\partial D^\pm$ in equation~\eqref{integralsoln} by taking a semicircular path around each zero, \emph{before} making the substitution for $\widetilde{f}$ and $\widetilde{g}$, obtaining convergent integrals. As shown in figure~\ref{fig:defn.Dtilde}, near each zero of $\Delta$, this  deformation to the contours of integration is such that each zero lying on the original contour is avoided.
\begin{figure}[h!]
  \begin{center}
    \caption{\label{fig:defn.Dtilde}
The new domains $\widetilde{D}^\pm$ in the case $|\alpha|=1$. The dashed lines represent $\partial D^\pm$, and the dots correspond to zeros of $\Delta$. Note that, as $0$ is a removable singularity of all integrands, we may make an arbitrary finite contour deformation near zero.}
    \includegraphics{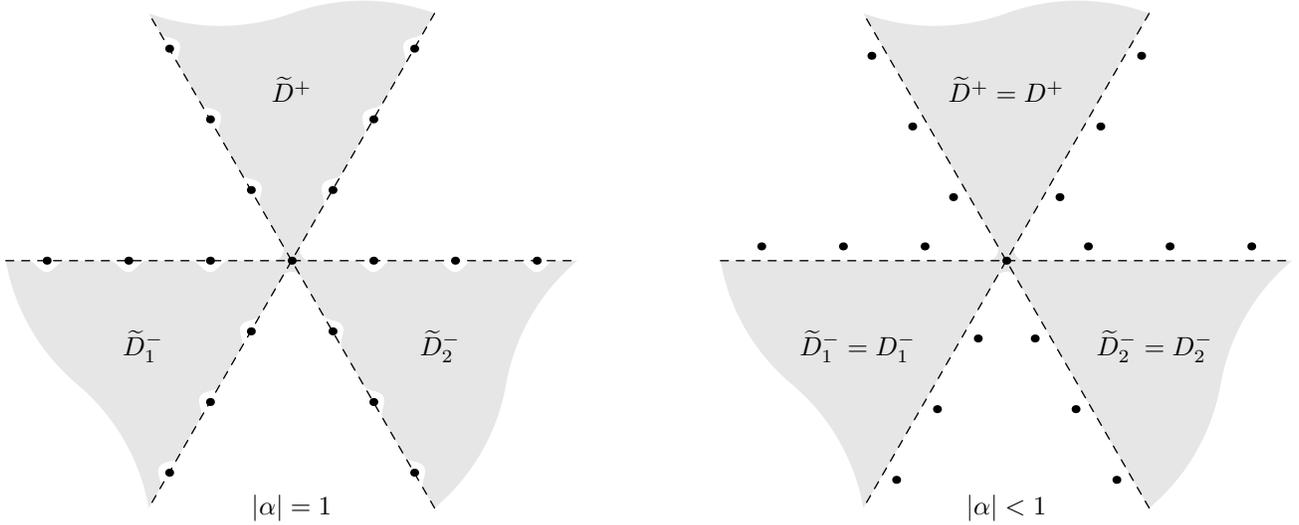}
  \end{center}
\end{figure}

\noindent Indeed, choosing $\widetilde{D}^\pm$ satisfying
\begin{equation} \label{eqn:defn.Dshrunk}
	\widetilde{D}^\pm = D^\pm \setminus \bigcup_{\substack{\lambda\in\mathbb{C}\setminus\{0\}:\\\Delta(\lambda)=0}} \mbox{ a small neighbourhood of } \lambda,
\end{equation}
by analyticity of $\widetilde{f}$ and $\widetilde{g}$, Cauchy's theorem and equation~\eqref{integralsoln}, it holds that
\begin{equation}\label{integralsoln.deformed}
	q(x,t)=\frac{1}{2\pi}\qc{\int_{\Re} e^{ikx+ik^3t} \, \widehat{q_{0}}(k) \, dk - \int_{\partial \widetilde{D}^{+}} e^{ikx+ik^3t} \, \widetilde{f}(k,t) \, dk  - \int_{\partial \widetilde{D}^{-}} e^{ik(x-L)+ik^3t} \, \widetilde{g}(k,t) \, dk}.
\end{equation}
Then, because of the way we chose $\widetilde{D}^{\pm}$, the modified statement
\begin{enumerate}
	\item[(B')]{
		Every zero of $\Delta$ is exterior to $\widetilde{D}^{\pm}$, i.e.\ strictly to the right of $\partial \widetilde{D}^\pm$
	}
\end{enumerate}
is immediate, and
\begin{enumerate}
	\item[(C')]{
		The part of the integrand excluding the first exponential factor, in the fourth (respectively, fifth) integral of
		equation~\eqref{eqn:SolnqWithqt}
		decays as $k\to\infty$ from within the closure of $\widetilde{D}^+$ (respectively, the closure of $\widetilde{D}^-$)
	}
\end{enumerate}
can be justified using an argument similar to that used for statement~(C). In the same way as for $|\alpha|<1$ above. The statements~(A),~(B') and~(C') then imply that the latter two integrals of equation~\eqref{eqn:SolnqWithqt} vanish.

Note that, as $\Delta$ satisfies the symmetry $\Delta(\tau k)=\Delta(k)$, we are free to choose $\widetilde{D}^\pm$ in such a way that $\widetilde{D}^+\cup\widetilde{D}^-$ has the same rotational symmetry. Further, as shown in figure~\ref{fig:defn.Dtilde}, in the case $|\alpha|<1$, by choosing small enough neighbourhoods of each nontrivial zero of $\Delta$, we can ensure $\widetilde{D}^\pm = D^\pm$.
Hence, for all $|\alpha|\leq1$, with $\widetilde{D}^\pm$ defined by equation~\eqref{eqn:defn.Dshrunk} using any choice of finite neighbourhoods, we can write
		equation~\eqref{eqn:SolnqWithqt} as
\begin{multline}
	2\pi q(x,t)=\int_{\Re} e^{ikx+ik^3t} \, \widehat{q_{0}}(k) \, dk \\
	- \int_{\partial \widetilde{D}^{+}} e^{ikx+ik^3t} \, k^{2}\widetilde{f_{0}}+\frac{N_{\mathrm{data}}(k) A_{1}(k)-\tau^2 N_{\mathrm{data}}(\tau k) A_{2}(k)-N_{\mathrm{data}}(\tau^2 k)A_{3}(k)} {\Delta(k)}\, dk \\
	- \int_{\partial \widetilde{D}^{-}} e^{ik(x-L)+ik^3t} \, k^{2}\widetilde{g_{0}}+\frac{N_{\mathrm{data}}(k) B_{1}(k)-\tau^2 N_{\mathrm{data}}(\tau k) B_{2}(k)-N_{\mathrm{data}}(\tau^2 k)B_{3}(k)}{\Delta(k)} \, dk.
	\label{usualrep}
\end{multline}
This representation depends only upon the data of the problem,  and  is the  usual representation provided in the Unified Transform literature for this problem.

This representation could be the beginning of a numerical strategy. However, in implementing the numerical evaluation, we found that a more symmetric, alternative formulation of  (\ref{usualrep}) is a more effective starting point.


\subsection{An alternative formulation of the integral representation} \label{ssec:AlternativeFormulation}
%

To rewrite the integral representation in a more elegant and symmetric way,  we combine integrands to define two new spectral functions, $ \zeta^{+}(k,t)$ and $ \zeta^{-}(k,t)$.

Define
\begin{align*}
\widetilde{f}_{data}(k,t) &= k^{2}\widetilde{f_{0}}+\frac{1}{\Delta(k)} \, \qc{N_{data}(k,t) A_{1}(k)-\tau^2 N_{data}(\tau k,t) A_{2}(k)-N_{data}(\tau^2 k,t)A_{3}(k)}, \\
\widetilde{g}_{data}(k,t) &= k^{2}\widetilde{g_{0}}+\frac{1}{\Delta(k)} \, \qc{N_{data}(k,t) B_{1}(k)-\tau^2 N_{data}(\tau k,t) B_{2}(k)-N_{data}(\tau^2 k,t)B_{3}(k)},
\end{align*}
Then set
\begin{equation}
  \zeta^{+}(k,t):= \widetilde{f}_{data}(k,t) \Delta(k), \quad \quad \zeta^{-}(k,t):= \widetilde{g}_{data}(k,t) \Delta(k)
\label{zetadef}\end{equation}


\begin{prop}\label{prop1}
The functions $\zeta^\pm$ defined by (\ref{zetadef}) satisfy the relation
\begin{equation}
\zeta^{+}(k,t)-e^{-ikL}\zeta^{-}(k,t)=\widehat{q_{0}}(k) \, \Delta(k),\qquad k\in\mathbb C.
\end{equation}
\end{prop}
\begin{proof}
Using the definition (\ref{NDelta}) of the function $N(k)$ and rearranging terms,
we find
\begin{equation}
  \begin{split}
\zeta^{+}(k,t)-e^{-ikL}\zeta^{-}(k,t)
&=
\bigg[\,\widehat{q_{0}}(k) (A_1-e^{-ikL}B_1)-\tau^2\widehat{q_{0}}(\tau k) (A_2-e^{-ikL}B_2)-\widehat{q_{0}}(\tau^2 k)(A_3-e^{-ikL}B_3)\\
&\qquad
-  k^2 \widetilde{f_{0}} \, [(A_1-e^{-ikL}B_1) -
 \tau (A_2-e^{-ikL}B_2)-\tau (A_3-e^{-ikL}B_3)-\Delta(k)]\\
&\qquad
+  k^2 \widetilde{g_{0}} \, [e^{-ikL} (A_1-e^{-ikL}B_1)
  \\
  &\qquad \qquad \qquad -\tau e^{-i\tau kL}(A_2-e^{-ikL}B_2)-\tau e^{-i\tau^2 kL}(A_3-e^{-ikL}B_3)
 -e^{-ikL}\Delta(k)]\bigg].
  \end{split}
\end{equation}
A straightforward calculation then shows we have that
\begin{equation}
  \begin{split}
    A_1-e^{-ikL}B_1 &= \Delta(k) \\
    A_2-e^{-ikL}B_2 &=0 \text{ and} \\
    A_3-e^{-ikL}B_3 &=0.
  \end{split}
\end{equation}
Hence
$$
\zeta^{+}(k,t)-e^{-ikL}\zeta^{-}(k,t)
=\widehat{q_{0}}(k) \, \Delta(k),
$$
as required.
\end{proof}
\noindent
Using the result of Proposition \ref{prop1}, we can simplify the evaluation of the integral representation solution equation \eqref{integralsoln}, and absorb the first integral into the other two. The  integral representation thus obtained for the solution of the problem is given by:
\begin{equation}\label{newintrep}
q(x,t)=\frac{1}{2\pi}\qc{ \int_{\partial \widetilde{E}^{+}} e^{ikx+ik^3t} \, \frac{\zeta^{+}(k,t)}{\Delta(k)} \, dk + \int_{\partial \widetilde{E}^{-}}e^{ik(x-L)+ik^3t}\, \frac{\zeta^{-}(k,t)}{\Delta(k)} \, dk\,}
\end{equation}
where  the domains  $\widetilde{E}^{+}$ and $\widetilde{E}^{-}$ are the half-plane complements of the domains $\widetilde{D}^+$ and $\widetilde{D}^-$ (see figure \ref{figE}), with the usual orientation of the integration path around $\widetilde{E}^{+}$ and $\widetilde{E}^{-}$, leaving the domain always on the left.
\begin{figure}[H]
  \centering
  \caption{\label{figE}
    {The complex regions $\widetilde{E}^{+}$ and $\widetilde{E}^{-}$. The boundary of these regions is where the integral (\ref{newintrep}) is evaluated.}}
  \includegraphics{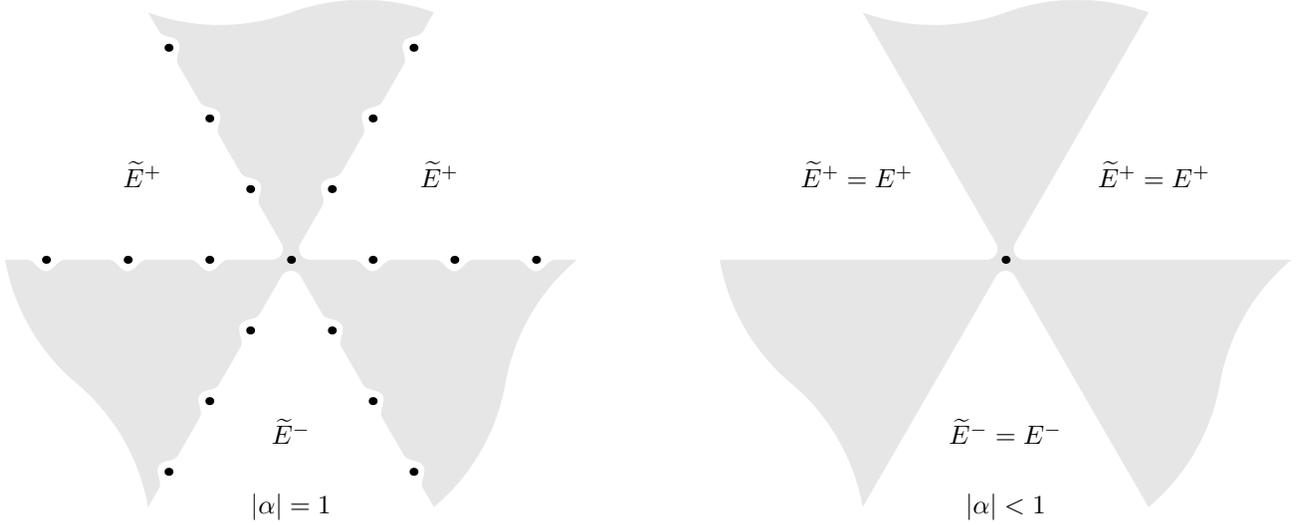}
\end{figure}

\subsection{Deformation of line integrals}





Figure~\ref{figE} suggests a particular choice of contours of integration $\partial \widetilde{E}^\pm$, for each value of $\alpha$. However, the arguments of sections~\ref{ssec:AnalysisGR}--\ref{ssec:AlternativeFormulation} do not require that any particular contour is followed, only that the chosen contour passes each nontrivial zero of $\Delta$ on the same side as does the corresponding contour on figure~\ref{figE}. Therefore, we are free to choose our contours in such a way that the numerical integration can be performed with optimal efficiency. In order to select appropriate contours, we study the asymptotics of the integrands involved.

As soon as any contour of integration has been deformed away from the rays of $\tau^j\mathbb{R}$ for $k\gg1$, the rapid decay (or blow-up) of the exponential factor $e^{ik^3t}$ determines the asymptotics. Therefore, as a general principal, it is advantageous to ``close'' the contours into $E$ and away from $D$. Indeed, by doing so, the tail of the integrals is bounded, and the error introduced by truncating the contours to a finite length is controlled. This procedure will be described further in section \ref{section4}.

\subsubsection*{The case $|\alpha|=1$}

In this case, the zeros of $\Delta$ are all on the lines $\tau^j\mathbb{R}$ for $j=0,1,2$, and are asymptotically regularly distributed, so they have a finite infimal separation. Moreover, as $k\to\infty$ within any subset of $\mathbb{C}^\pm$ bounded uniformly away from each zero of $\Delta$, the ratio $\zeta^\pm/\Delta=\mathcal{O}(|k|^{-1})$. Therefore, using Jordan's lemma and the decay of the exponential factor $e^{ik^3t}$, we can deform the contours of integration as far as we choose into $E^\pm$, provided we leave loops about each nontrivial zero of $\Delta$.

Each such loop integral corresponds to a residue at the pole.
As the spatial differential operator associated with this problem is self-adjoint, its eigenfunctions form a basis in the space of admissible initial data, so the solution may be represented as a convergent expansion in the eigenfunctions. By~\cite{FS2016a,Bir1908b}, the nontrivial zeros of $\Delta$ are the cube roots of the eigenvalues, and the series obtained by evaluating the corresponding residues is the eigenfunction expansion. As this series converges, the error introduced by truncating the series is controlled.

Using the above argument, we may ignore the contributions of all but a few residues close to $0$. So, for some $N\in\mathbb{N}$, we choose infinite contours which:
\begin{itemize}
	\item{
		Lie within $D$ close to $0$, so that they enclose the $N$ poles of the integrand that lie closest to $0$ on each ray, and the contributions from these poles are included,
	}
	\item{
		Cross into $E$ after the $N$-th but before the $N+1$-st pole on each ray, so that the contours do not pass through any of the poles,
	}
	\item{
		Extend to $\infty$, but remain bounded within $E$, thereby excluding the small contribution from the remaining poles, but ensuring the integrand is rapidly decaying along this part of the contour.
	}
\end{itemize}
Finally, we truncate the contour soon after it enters $E$, as the contribution from the remaining infinite part is very small.

\subsubsection*{The case $0<|\alpha|<1$}

For these values of $\alpha$, the same asymptotic results as in the $|\alpha|=1$ case hold on the decay of $\zeta^\pm/\Delta$, but now the zeros of $\Delta$ lie asymptotically on rays that lie within $E$ (see figure \ref{fig:defn.Dtilde}) and parallel to the lines $\tau^j\mathbb{R}$ for $j=0,1,2$. The distance between these rays and lines is a strictly increasing function of $1/|\alpha|$. As the poles are now uniformly bounded inside $E^\pm$, we may choose to deform our contours of integration into $E^\pm$ in such a way that they are still straight rays, and remain on the same side of each zero of $\Delta$, but we can take advantage of the decay of the exponential factor $e^{ik^3t}$. In this way, no analysis of the residues is necessary, but we may still truncate the contours without introducing a large error.

\subsubsection*{The case $\alpha=0$}

The zeros of $\Delta$ now lie on the rays $-i\tau^j\mathbb{R}^+$, for $j=0,1,2$, but the ratios $\zeta^\pm/\Delta$ \emph{do not decay} as $k\to\infty$ from within $E^\pm$, even if $k$ is uniformly bounded away from the zeros of $\Delta$. Certainly, the integrand grows exponentially as $k\to\infty$ from within $D^\pm$, because of the $e^{ik^3t}$ factor. Therefore, it is not possible to deform the contours of integration away from the lines $\tau^j\mathbb{R}$ at infinity. Of course, we can still make any finite deformation as long as we do not cross any nontrivial zeros of $\Delta$.

However, as $k\to\infty$ along the rays $\tau^{j/2}\mathbb{R}^+$ for $j=0,1,2,3$ (respectively, $j=4,5$), $\zeta^+/\Delta=\mathcal{O}(|k|^{-1})$ (respectively, $\zeta^-/\Delta=\mathcal{O}(|k|^{-1})$). Therefore, using the Riemann-Lebesgue lemma, we can bound the error introduced by truncating the contours to a finite length.

\section{General numerical strategy and a comparison with the half-line problem}\label{section3}

The numerical strategy we propose to compute the unified transform method solution involves calculating line integrals defined on the complex $k$-plane. The numerical evaluation of line integrals containing an exponential term with an analytic function in the integrand involves the implementation of parabolic, hyperbolic or cotangent contour specifications given in \cite{Trefethen2006}. By specifying line integrals around $\partial \widetilde{E}^\pm$ as hyperbola may numerically approximate the solution to the equation~\eqref{lkdv} for $0 \leq \alpha < 1$. When $\alpha=1$, a new kind of contour specification is required since all the poles are on the $\tau^j\mathbb{R}$ lines for $j=0,1,2$.

After implementing the deformation mappings $k(\theta)$ to the integral representations given by~\eqref{newintrep}, the line integrals on the complex value $k$ become real valued integrals on $\theta \in [-\infty, \infty]$. In application, we truncate the values $\theta$ can take by selecting a $\theta_{max}$ such that $\theta \in [-\theta_{max}, \theta_{max}]$. We quantify the selection of $\theta_{max}$ and benchmark our results by implementing a forcing term to Airy's equation. This allows us to choose a specific forcing function so that $q(x,t)$ is known.

The Fokas transform method has been applied to the half-line problem for Airy's equation in \cite{flyerfokas2008}. In order to quanitify the accuracy of simulations to the full BVP, we take the opportunity to benchmark the deformation approximation according to different $\theta_{max}$ values on the half-line.
Also, the comparison of half-line problem and finite interval case will be given in this section.

\subsection{Half-line case}
{Consider the non-homogeneous initial boundary value problem for Airy's equation
\begin{equation}\label{nonhomLKdV}
\begin{split}
  q_{t}(x,t)+q_{xxx}(x,t)&=h(x,t), \quad x\in[0,\infty), \quad t>0, \\
  q(x,0)&=q_{0}(x),\\
  q(0,t)&=f_{0}(t).
  \end{split}
\end{equation}
The Fokas integral representation solution of the \eqref{nonhomLKdV} can be computed as
\begin{equation}\label{halflinesoln}
q(x,t)=\frac{1}{2\pi} \qc{\int_{\Re} e^{ikx+ik^3t} \, \widehat{q_{0}}(k) \, dk - \int_{\partial D^{+}} e^{ikx+ik^3t}\, \frac{\zeta^{+}(k,t)}{\Delta(k)}\, dk + \int_{\Re} e^{ikx+ik^3t} \, H(k,t) \, dk}
\end{equation}
where the {\em spectral functions} $\widehat{q_{0}}(k)$ and $\zeta^{+}(k,t)$ are defined in (\ref{eq:qhat}) and (\ref{zetadef}) and $H(k,t)$ is given by
\begin{equation}\label{Hdef}
H(k,t)= \int_{0}^{t}\int_{0}^{\infty} e^{-ikx-ik^3s}\, h(x,s)\, dx ds.
\end{equation}
Note that the original (not the symmetric form of) the Fokas integral representation is used for the solution representation of the problem in the half-line. Due to the nature of the problem on the half-line there is no boundary value on the right which corresponds line integral around $D^{-}$ region. For the detailed analysis of the Fokas integral representation solution of non-homogeneous linear evolution PDEs, see~\cite{papatheodoroukandili2009}.}

\begin{example}
  \label{ex:halfline}
  We select the initial, boundary and forcing conditions such that
  \begin{eqnarray*}
    q_{0}(x)&=&0 \\
    f_{0}(t)&=&\sin{2 \pi t} \\
    h(x,t)&=&2 \pi \,e^{-x}\cos{(2 \pi t)}-e^{-x}\sin{(2 \pi t)},
  \end{eqnarray*}
  then the analytic solution of non-homogenous Airy's equation defined by~\eqref{nonhomLKdV} is
  \begin{equation*}
    q(x,t) = e^{-x} \,\sin{(2 \pi t)}.
  \end{equation*}
\end{example}
It is important to note that in order to evaluate line integrals along $\Re$ and $\partial D^{+}$, we can deform all of the line integrals with the same hyperbola contour specification  $k(\theta)=i\sin{(\frac{\pi}{6}-i\theta)}$. By the nature of the problem, there are no poles in the half-line case.
We implement $k(\theta)$ as a deformation mapping for the solution given by equation~\eqref{halflinesoln}. The results illustrating the effect of varying $\theta_{max}$ is given in figure~\ref{halflinebenchmark}. As expected, the error introduced by truncating the contour decreases when $\theta_{max}$ increases since the deformation mapping stays in the region where both $e^{ikx}$ and $e^{ik^3t}$ are bounded.

\begin{figure*}[h!]
    \centering
    \caption{A numerical benchmark of Airy's equation posed on the half-line with $(x,t) \in [0,\infty)\times[0,1]$, see example \ref{ex:halfline}. Note that the error induced by truncating the deformation contour becomes very small very quickly as $\theta_{max}$ is increased.}\label{halflinebenchmark}
    \begin{subfigure}[t]{0.3\textwidth}
        \centering
        \includegraphics[height=0.9\textwidth]{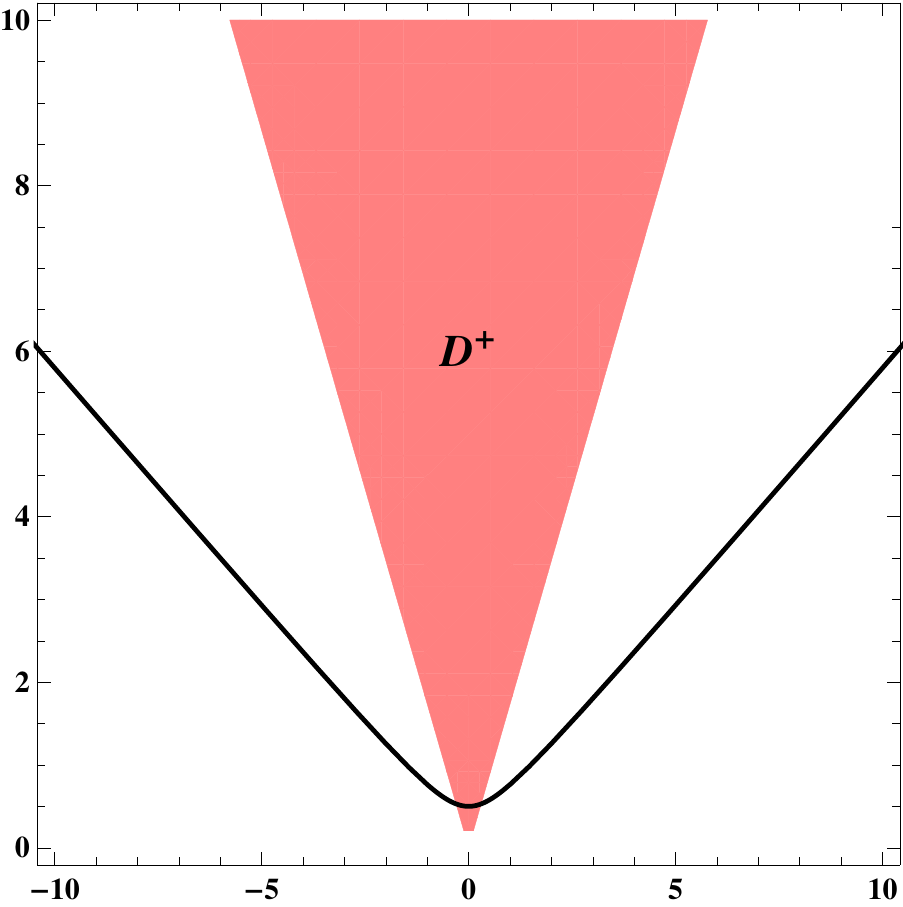}
        \caption{The computation region and the deformation mapping.}
    \end{subfigure}%
    \hfill
    \begin{subfigure}[t]{0.3\textwidth}
        \centering
        \includegraphics[height=0.9\textwidth]{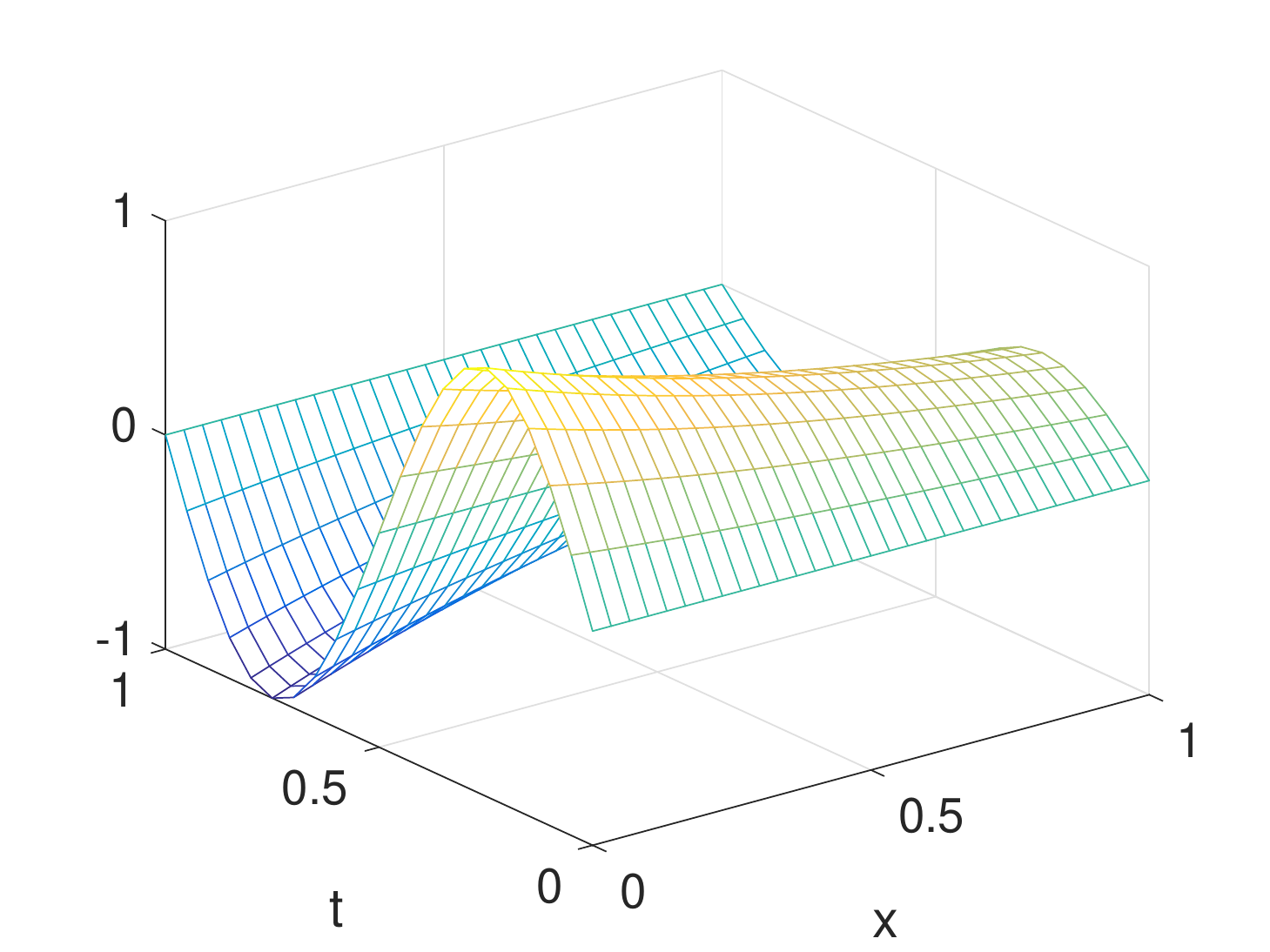}
        \caption{The numerical approximation with $\theta_{max}=50$.}
    \end{subfigure}
    \hfill
    \begin{subfigure}[t]{0.3\textwidth}
      \centering
      \includegraphics[height=0.9\textwidth]{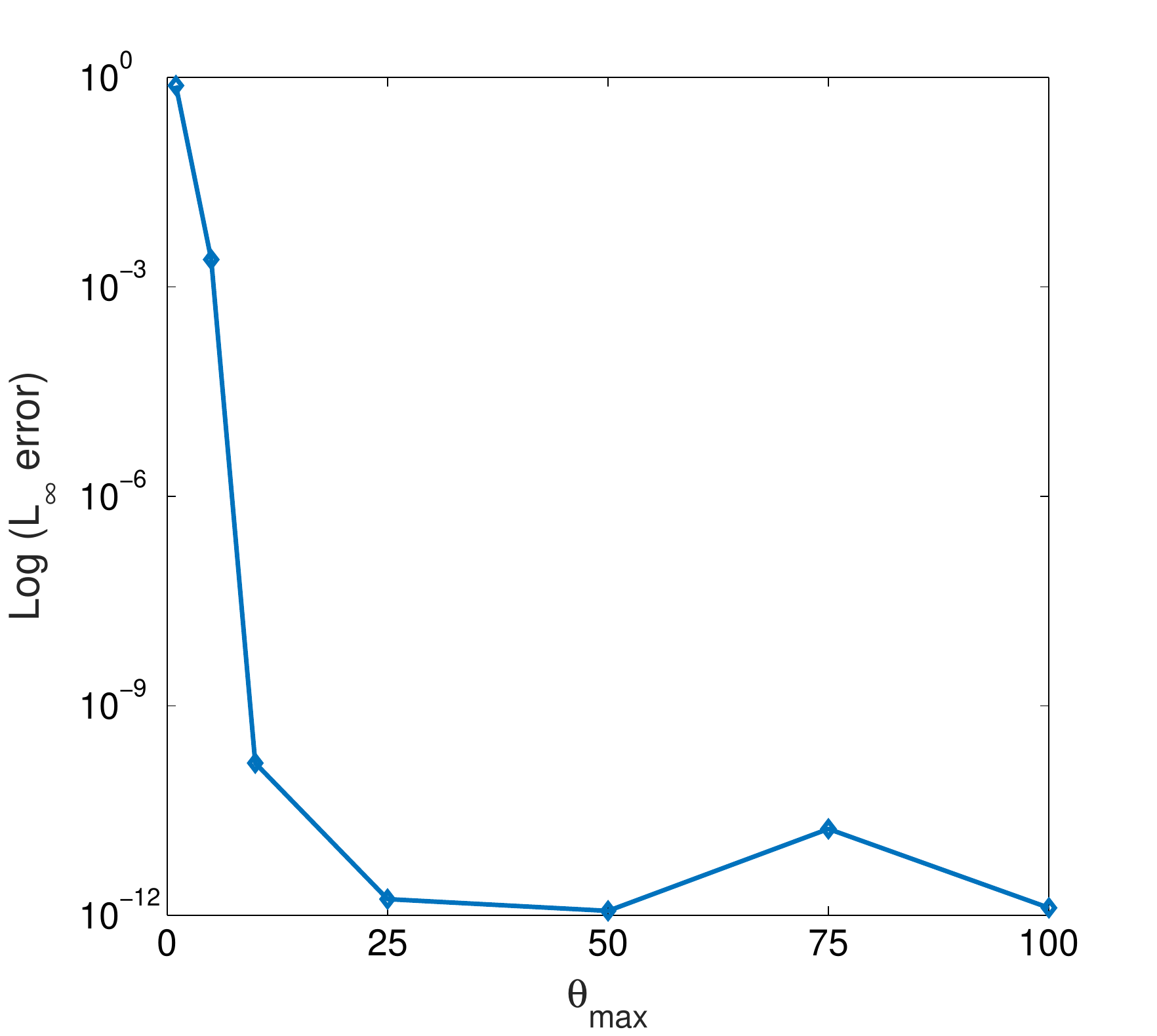}
      \caption{A plot of the maximum error in approximation against the contour truncation value.}
    \end{subfigure}%
\end{figure*}

\subsection{Finite interval case}
The computation of the solution is far more complicated when Airy's equation is posed over a finite interval $[0,L]$ than when it is posed on the half-line. Here, choosing an appropriate deformation mapping is crucial to evaluate the Fokas integral representation solution (\ref{newintrep}). For the problem~\eqref{lkdv} we require that deformation mappings $k(\theta)$ must have the following properties:
\begin{enumerate}[i)]
\item{The deformed contours should not pass through the poles which are zeros of $\Delta(k)$.}
\item{The deformed contours should follow a trajectory such that any poles stay on the left side of the direction of these mappings.}
\item{When $0 \leq \alpha < 1$, the deformed contours must remain asymptotically on the $\tau^j\mathbb{R}$ lines for $j=0,1,2$. Indeed, any deformation from these lines results in at least one of the exponentials $e^{\pm\tau ^jk}$, involved in the integrand, to be unbounded for large $|k|$, 
see e.g. figure~\ref{tikz}.
    \begin{figure}[H]
    \centering
    \caption{\label{tikz} \small{Deformation effects on $e^{ikL}$, $e^{i\tau kL}$ and $e^{i\tau^2 kL}$ terms.}}
    \includegraphics[height=0.4\textwidth,keepaspectratio=true]{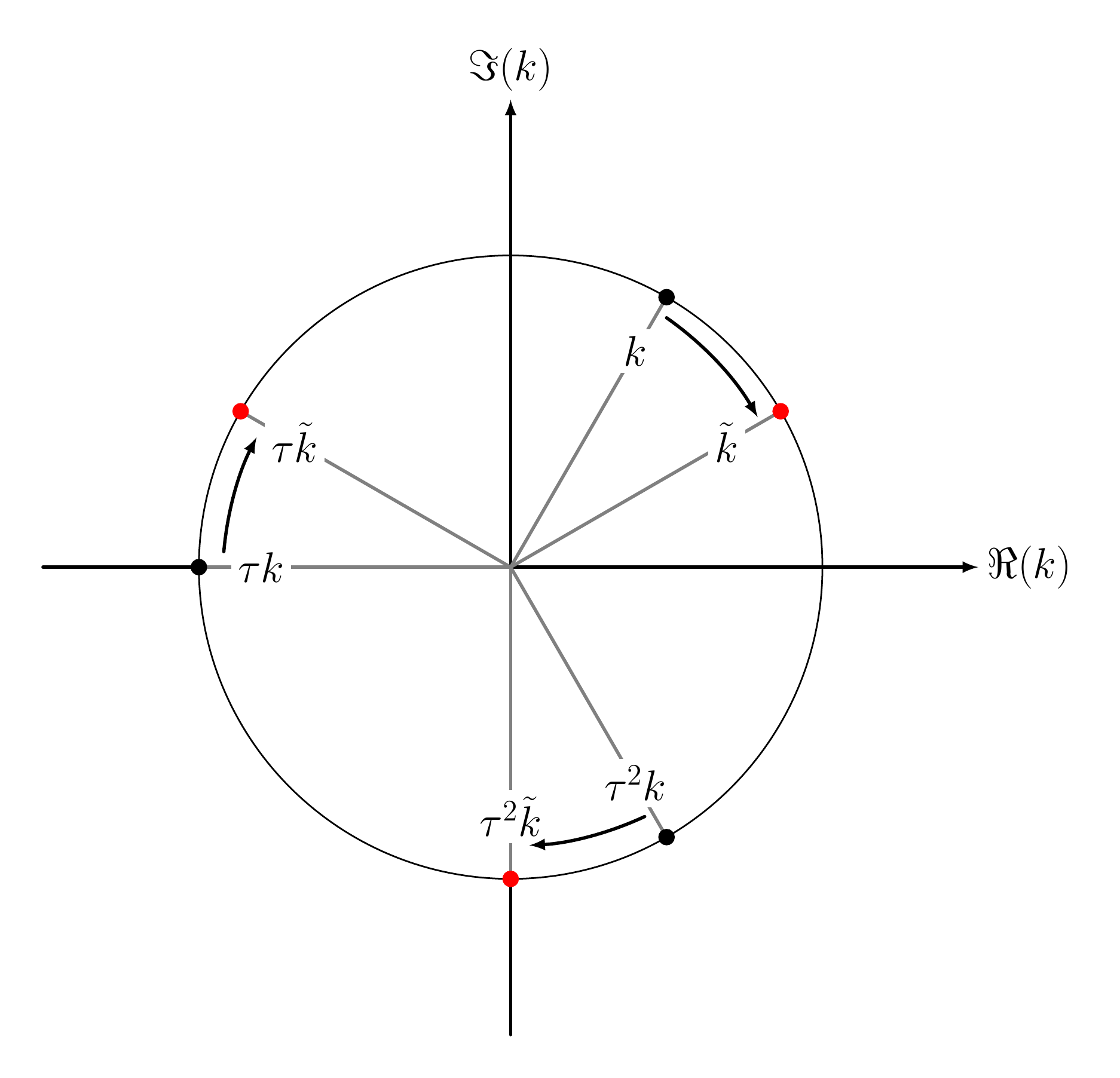}
    \vspace{-10pt}
    \end{figure}
}
\item{The numerical  integral value should include the contributions of the residues of sufficiently many of the poles lying on the $\tau^j\mathbb{R}$ lines ($j=-1,2,$) which are close to $0$. When $\alpha=1$, this can be done by deforming the contours so that they remain within the $D^{\pm}$ region near $0$ and in $E^{\pm}$ otherwise, see figure \ref{fig:accountforresidue}.}
\end{enumerate}

\begin{figure}[h!]
  \centering
  \caption{An example of a deformed contour that accounts for information of residues near the origin.
  }
  \includegraphics[height=0.4\textwidth]{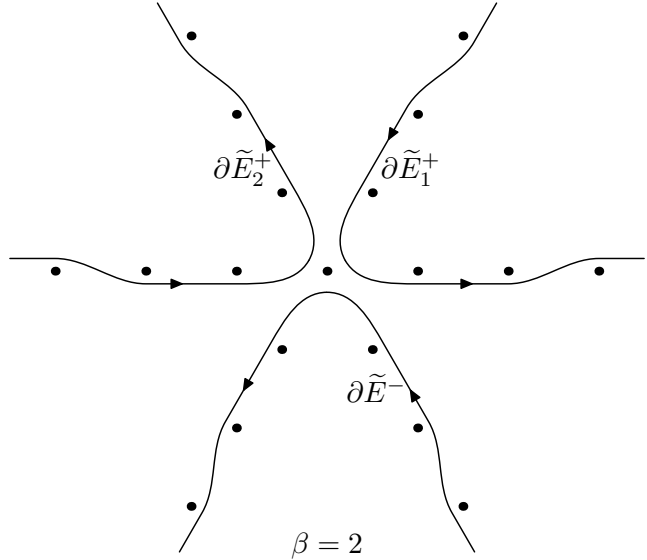}
  \label{fig:accountforresidue}
\end{figure}

\section{Numerical examples}\label{section4}
In this section we summarise extensive numerical experiments for Airy's equation defined on a finite interval $[0,1]$. We study the effects of truncating the deformation contour and varying the boundary coupling constants $\alpha \in [0,1]$. Different values of {$\alpha$} result in different positions of the poles. Hence we take care to vary the deformation functions accordingly. With that in mind we present the results in separate cases, $\alpha=0$, $0<\alpha<1$ and $\alpha=1$.

\subsection{When $\alpha=0$}
In order to approximate the numerical solution of the initial-boundary value problem~\eqref{lkdv} with $\alpha=0$, the complex contours are deformed by the following functions:
\begin{equation}\label{zerodeform}
  k_{1}(\theta)= \tau^2 k_{3}(\theta), \quad
  k_{2}(\theta)= \tau k_{3}(\theta), \quad
  k_{3}(\theta)=-i \eta \sin{\qp{\frac{\pi}{3}-i\theta}},
\end{equation}
where $k_{1}(\theta)$ and $k_{2}(\theta)$ deform the boundary of the domain $E^{+}$ and $k_{3}(\theta)$ deforms the boundary of $E^{-}$ (see figure \ref{figE}). The parameter $\eta$ determines the distance of the curve with respect to origin. To test the reliability of these deformation functions for $\alpha=0$ and to justify the truncation of $\theta$ to a finite region we compute the error by implementing a forcing term in Airy's equation similarly to our test in the half line (compare with example \ref{ex:halfline}), i.e.,
\begin{equation}\label{lkdv2}
  \begin{split}
    q_{t}(x,t)+q_{xxx}(x,t)&=h(x,t), \quad x\in[0,1], \quad t>0 \\
    q(x,0)&=q_{0}(x), \quad x\in[0,1]\\
    q(0,t)&=f_{0}(t), \quad t>0\\
    q(L,t)&=g_{0}(t), \quad t>0\\
    q_{x}(L,t)&=\alpha \, q_{x}(0,t), \quad t>0, \alpha\in\mathbb R.
  \end{split}
\end{equation}
This then allows us to specify a known $q(x,t)$ function as an analytic solution and benchmark appropriately.

The integral representation of the equation~(\ref{lkdv2}) is given by:
\begin{equation}
  q(x,t)=\frac{1}{2\pi} \qc{\int_{\Re} e^{ikx+ik^3t} \, [\,\widehat{q_{0}}(k)\, +\, H(k,t)\,] \, dk - \int_{\partial \widetilde{D}^{+}} e^{ikx+ik^3t} \, \frac{\zeta^{+}(k,t)}{\Delta(k)} \, dk - \int_{\partial \widetilde{D}^{-}}e^{ik(x-L)+ik^3t}\, \frac{\zeta^{-}(k,t)}{\Delta(k)} \, dk}
\end{equation}
where the spectral functions $\widehat{q_{0}}(k)$, $\zeta^{\pm}(k,t)$ and $H(k,t)$ are defined in (\ref{eq:qhat}), (\ref{zetadef}) and (\ref{Hdef}),   respectively. It should be noted that adding the forcing function to the equation adds one term to the function $N_{\mathrm{data}}$ which changes $\widetilde{f}_{data}$ and $\widetilde{g}_{data}$ functions, and thus the definition of $\zeta^{\pm}(k,t)$ function given in (\ref{zetadef}) is updated accordingly.
\begin{equation*}
N_{\mathrm{data}}(k) = \widehat{q_{0}}(k)-k^2 \widetilde{f_{0}} + e^{-ikL} k^2 \widetilde{g_{0}} + H(k)
\end{equation*}
Moreover, the integral from $0$ to $\infty$ in definition of the function $H(k,t)$ given in (\ref{Hdef}) should be replaced with the integral from $0$ to $1$, since $x \in [0,1]$.

The integral representation solution with forcing term can be also simplified to the symmetric form by absorbing the first integral into the other two integrals, with the similar argument done above. Then, the alternative formulation of the integral representation with forcing term is obtained by:
\begin{equation}
\label{eq:integral-rep-inc-forcing}
q(x,t)=\frac{1}{2\pi}\qc{ \int_{\partial \widetilde{E}^{+}} e^{ikx+ik^3t} \, \frac{\zeta^{+}(k,t)}{\Delta(k)} \, dk + \int_{\partial \widetilde{E}^{-}}e^{ik(x-L)+ik^3t}\, \frac{\zeta^{-}(k,t)}{\Delta(k)} \, dk\,}
\end{equation}
where  the domains  $\widetilde{E}^{+}$ and $\widetilde{E}^{-}$ are the half-plane complements of the domains $\widetilde{D}^+$ and $\widetilde{D}^-$. 

\begin{example}
  \label{ex:2}
  We select initial, boundary and forcing conditions such that
  \begin{eqnarray*}
    q(x,0)&=&0 \\
    q(0,t)&=&0 \\
    q(1,t)&=&\sin{\qp{2\pi t}}\\
    q_{x}(1,t)&=&0 \\
    h(x,t)&=&2\pi\,(2x-x^2)\cos{\qp{2\pi t}},
  \end{eqnarray*}
  then the analytic solution of the non-homogeneous Airy's equation,
  is given by
  \begin{equation*}
    q(x,t) = (2x-x^2) \,\sin{(2\pi t)}.
  \end{equation*}
\end{example}
We implement these initial, boundary and forcing terms into (\ref{eq:integral-rep-inc-forcing}) and make use of the contour deformations specified in (\ref{zerodeform}). We vary the truncation of the integral and illustrate the behaviour of the approximation in figure~\ref{ex2}. It is important to note that the truncation value of $\theta_{max}$ for $\alpha=0$ is far less than the half-line case. It is the existence of poles in finite interval case that restricts numerical computation for bigger $\theta_{max}$ values.

\begin{figure*}[!ht]
  \caption{A numerical benchmark of Airy's equation posed on a finite interval (\ref{lkdv2}). Initial, boundary and forcing functions are given in example \ref{ex:2}. Here $\alpha = 0$ and $\eta=1/2$. We test the effect of varying the truncation of the contour integrals in (\ref{eq:integral-rep-inc-forcing}) with $(x,t)\in[0,1]\times[0,1]$. Notice that as the contour length increases the error decreases quickly, as expected.}\label{ex2} 
    \begin{subfigure}[t]{0.5\textwidth}
        \centering
        \includegraphics[height=0.7\textwidth]{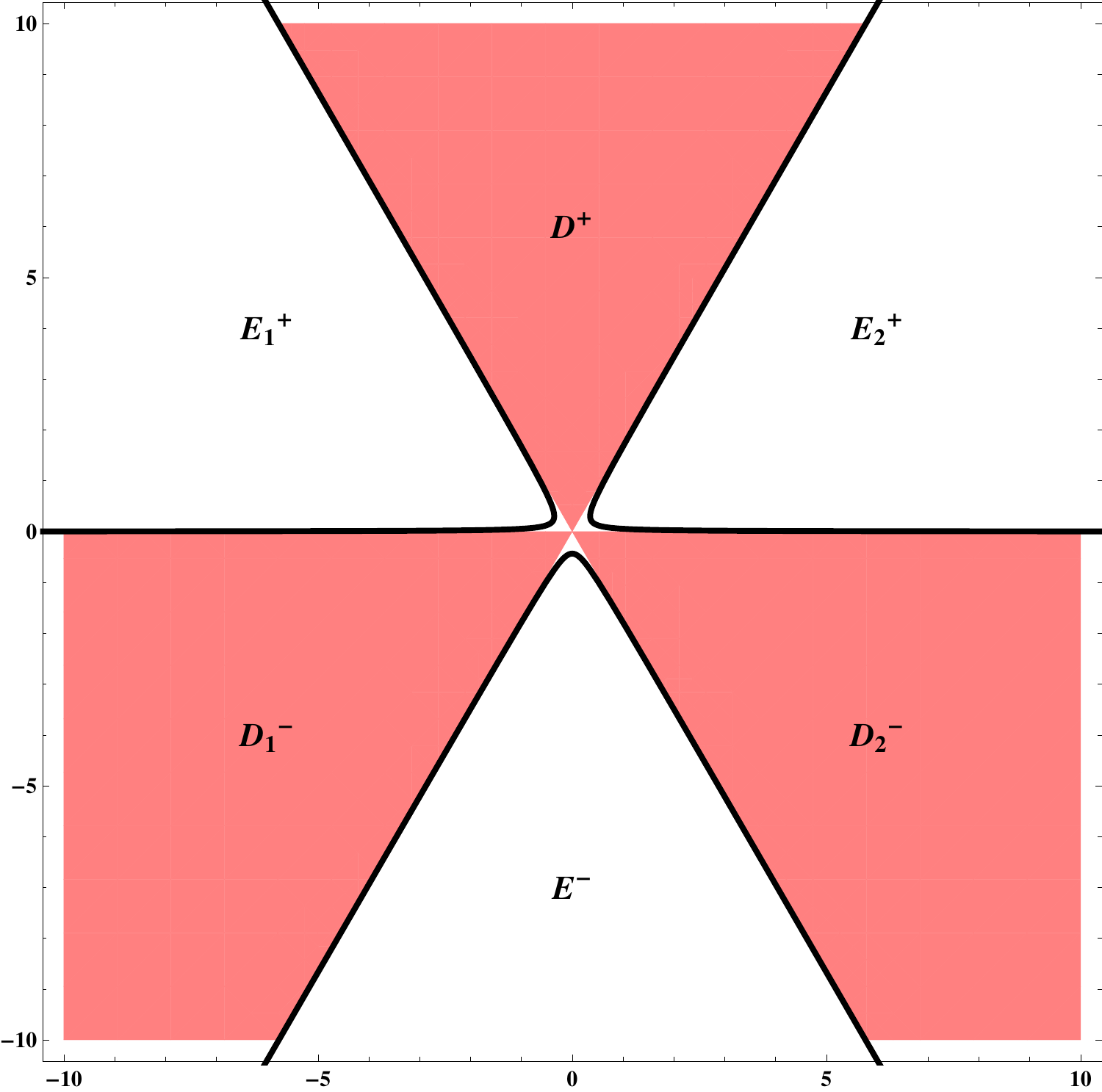}
        \caption{The computational region and the deformation mappings.} 
    \end{subfigure}%
    ~
    \begin{subfigure}[t]{0.5\textwidth}
        \centering
        \includegraphics[height=0.7\textwidth]{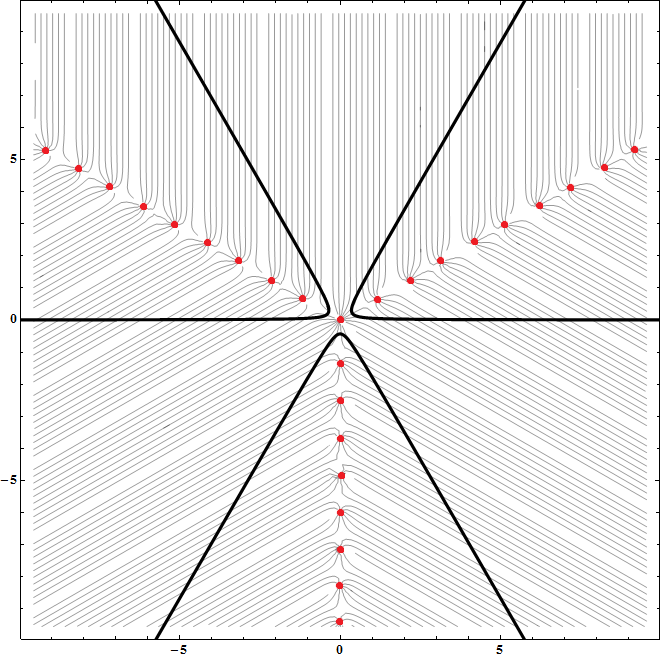}
        \caption{The location of the poles with respect to the deformation mapping.}
    \end{subfigure}
    \\
    \centering
    \begin{subfigure}[t]{0.5\textwidth}
        \centering
        \includegraphics[height=0.7\textwidth]{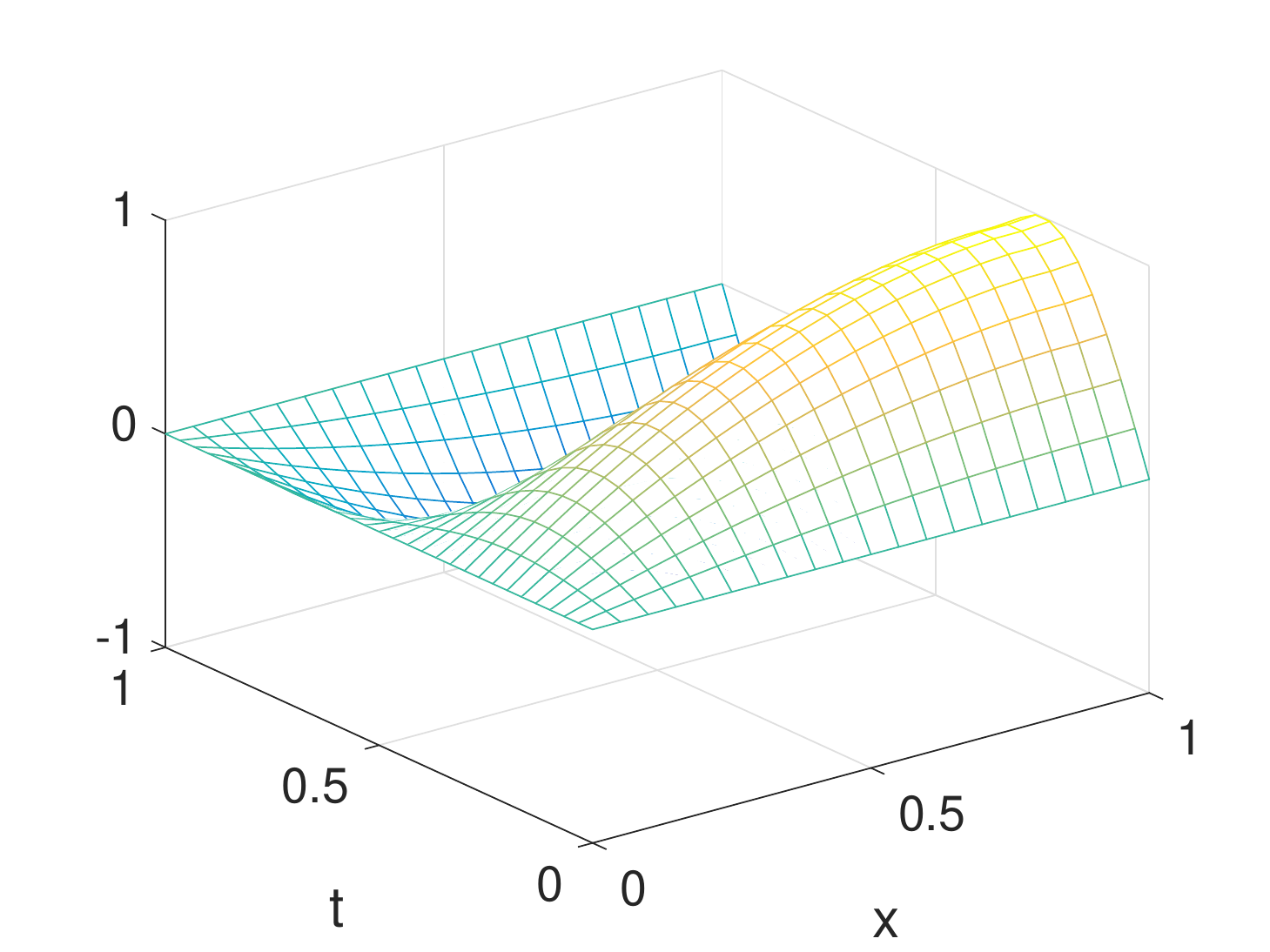}
        \caption{The numerical approximation with $\theta_{max} = 8$.}
    \end{subfigure}%
    ~
    \begin{subfigure}[t]{0.5\textwidth}
        \centering
        \includegraphics[height=0.7\textwidth]{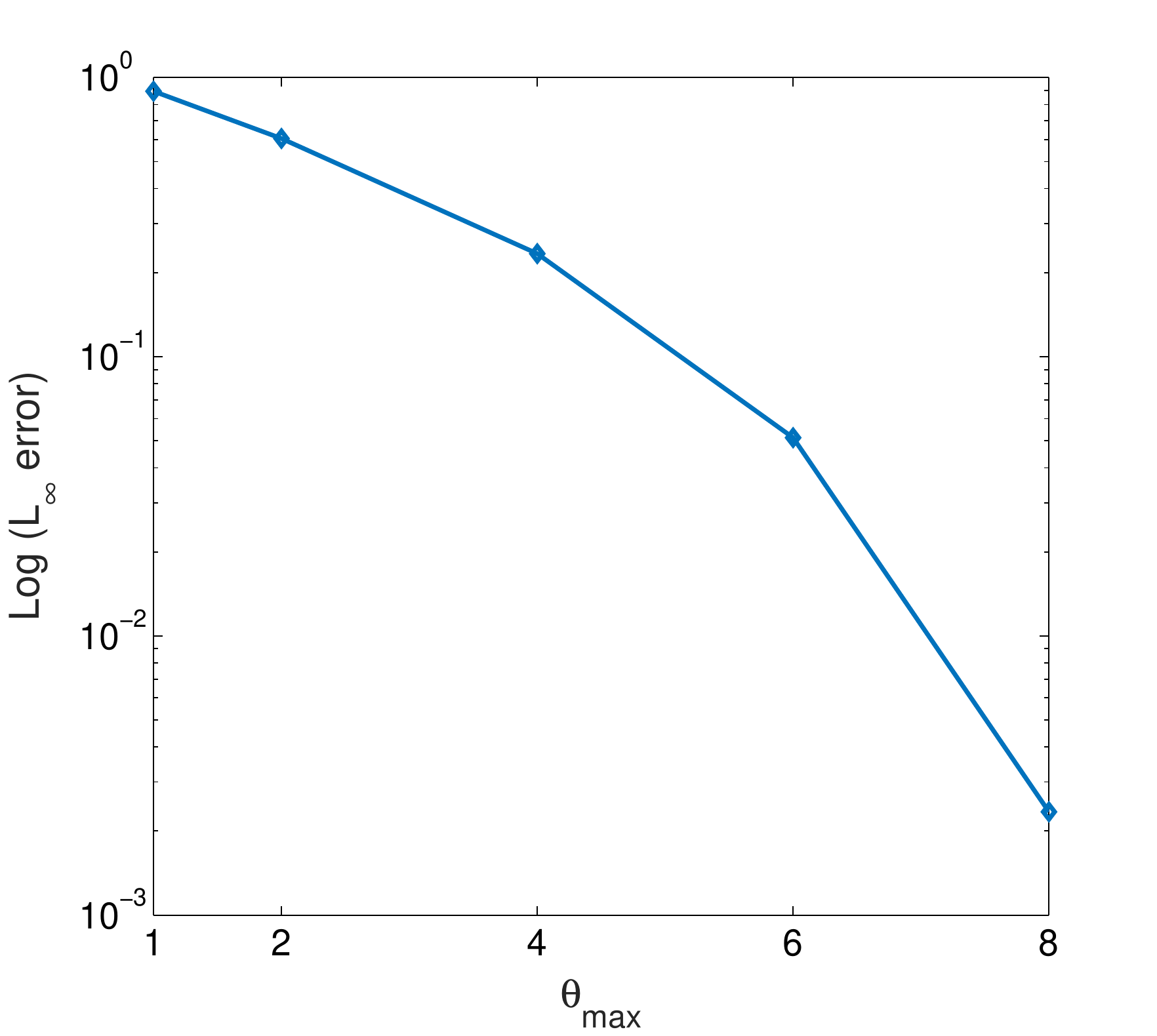}
        \caption{A plot of the maximum error in approximation against the contour truncation value.}
    \end{subfigure}
\end{figure*}

\clearpage
\begin{example}
  \label{ex:3}
  In this example we show the approximate solution for two different choices of initial and boundary conditions. Note in both cases there is \emph{no forcing term}. In the first case we select initial and boundary conditions such that
  \begin{equation}
    \label{eq:alpha0eq1}
    \begin{split}
      q(x,0)&=0\\
      q(0,t)&=\sin{\qp{2 \pi t}}\\
      q(1,t)&=0 \\
      q_{x}(1,t)&=0.
    \end{split}
  \end{equation}
\end{example}
The numerical approximation computed by the deformations defined for $\alpha=0$ above is given in figure \ref{fig:alpha0fig1}.

In the second case we select initial and boundary conditions such that
\begin{equation}
  \label{eq:alpha0eq2}
  \begin{split}
    q(x,0)&=0\\
    q(0,t)&=\sin{\qp{2 \pi t}} \\
    q(1,t)&=\sin{\qp{2 \pi t}} \\
    q_{x}(1,t)&=0.
  \end{split}
\end{equation}
The numerical approximation computed by the deformations defined for $\alpha=0$ above is given in figure \ref{fig:alpha0fig2}.
  \begin{figure}[H]
    \caption{\label{ex2sol} Examples of numerical approximations to Airy's equation using the Fokas transform method. }
    \centering
    \begin{subfigure}[t]{0.5\textwidth}
      \centering
      \includegraphics[height=0.7\textwidth,keepaspectratio=true]{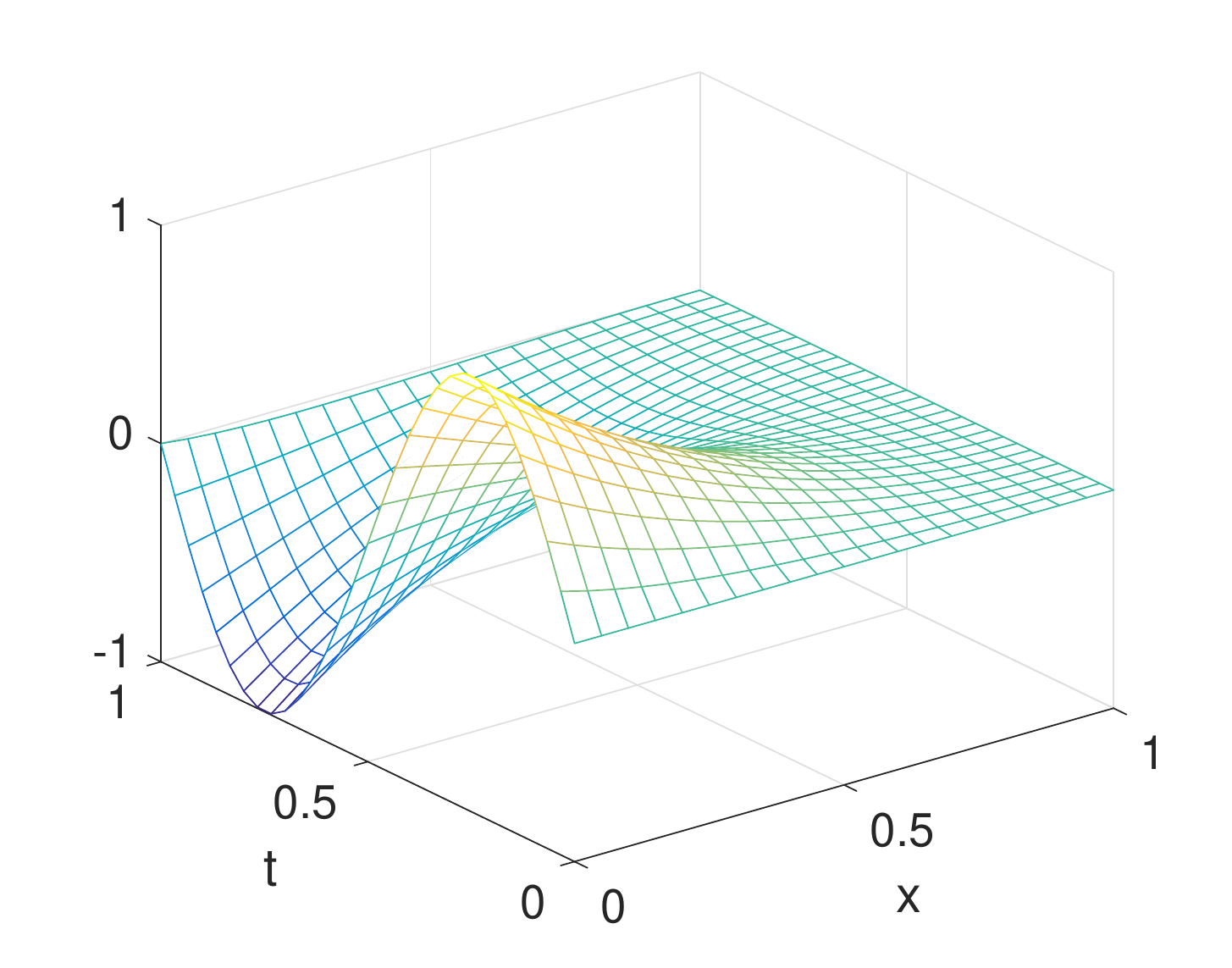}
      \caption{\label{fig:alpha0fig1} The solution given by the initial conditions (\ref{eq:alpha0eq1})}
    \end{subfigure}%
    \begin{subfigure}[t]{0.5\textwidth}
      \centering
      \includegraphics[height=0.7\textwidth,keepaspectratio=true]{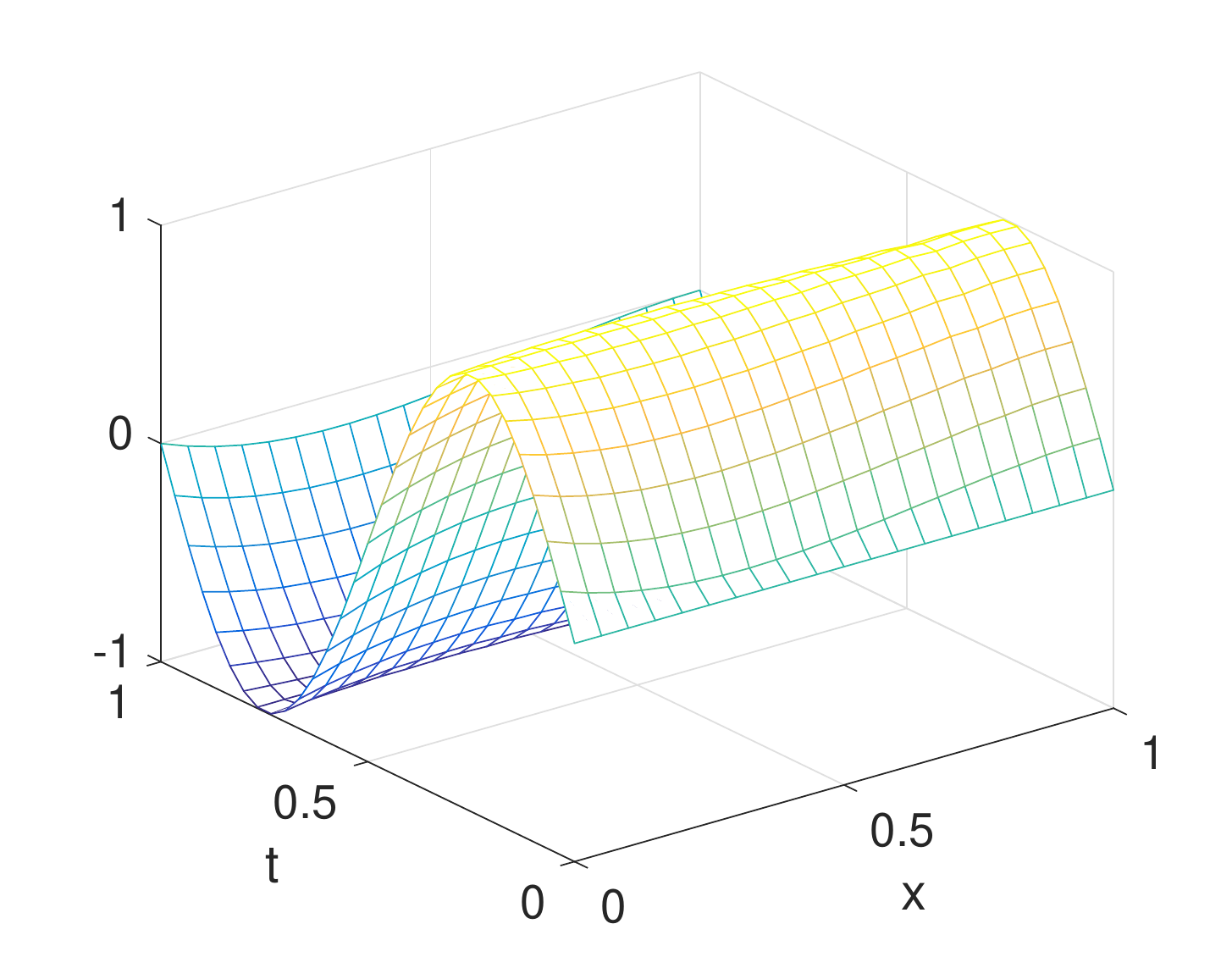}
      \caption{\label{fig:alpha0fig2} The solution given by the initial conditions (\ref{eq:alpha0eq2}).}
    \end{subfigure}%
  \end{figure}

\subsection{When $0<\alpha<1$}
As can be seen from figure \ref{fig:pole-location-1} the position of the poles changes when $\alpha$ varies between zero and one. We choose the complex contour deformations defined by the following functions:
\begin{equation}\label{deformex1}
k_{1}(\theta)= \tau^2 k_{3}(\theta), \quad
k_{2}(\theta)= \tau k_{3}(\theta), \quad
k_{3}(\theta)=-i \eta \sin{\qp{\frac{\pi}{3}-i\theta}}
\end{equation}
where $k_{1}(\theta)$ and $k_{2}(\theta)$ deform the boundary of the domain $E^{+}$ and $k_{3}(\theta)$ deforms the boundary of $E^{-}$ (see figure \ref{figE}). The $\eta$ parameter determines the distance between the curve and the origin.

\begin{figure}[ht!]
  \caption{
    \label{fig:pole-location-1}
    The location of the poles for the solution to Airy's equation for some specific $\alpha \in (0,1)$ with $\eta=1/2$.} 
    \centering
    \begin{subfigure}[t]{0.5\textwidth}
        \centering
        \includegraphics[height=0.7\textwidth]{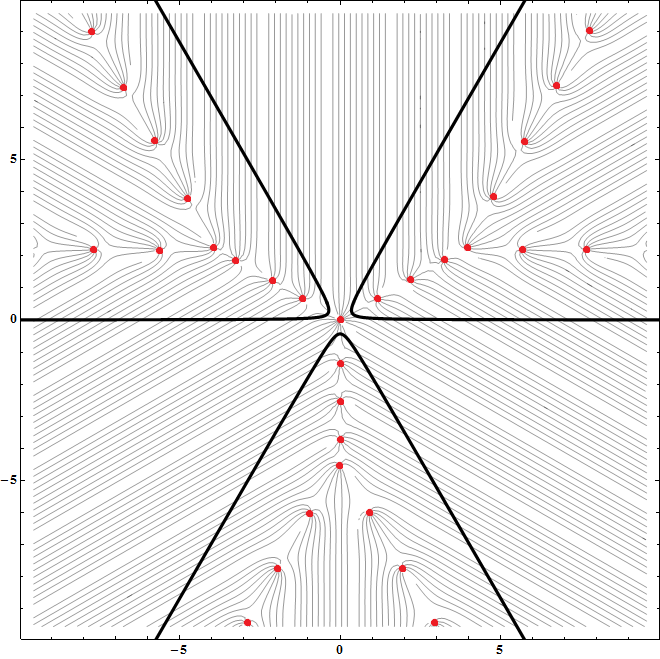}
        \caption{$\alpha = 0.001$}
    \end{subfigure}%
    ~
    \begin{subfigure}[t]{0.5\textwidth}
        \centering
        \includegraphics[height=0.7\textwidth]{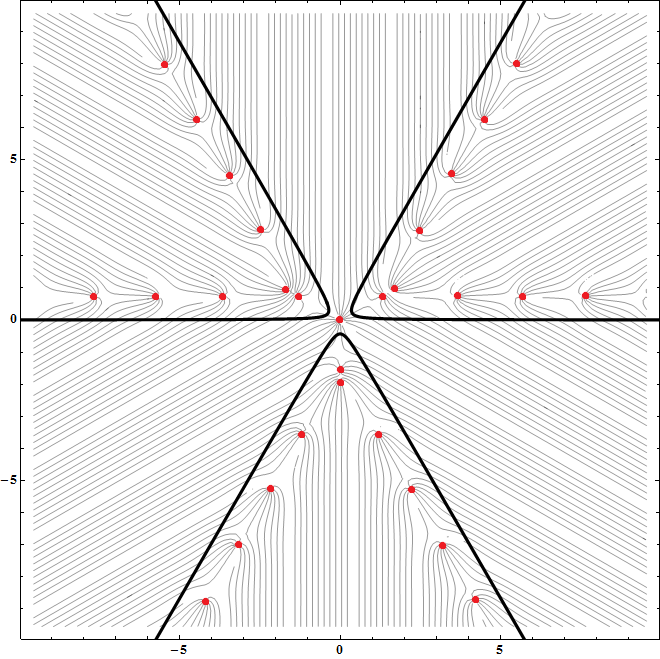}
        \caption{$\alpha = 0.1$}
    \end{subfigure}
    \\
    \centering
    \begin{subfigure}[t]{0.5\textwidth}
        \centering
        \includegraphics[height=0.7\textwidth]{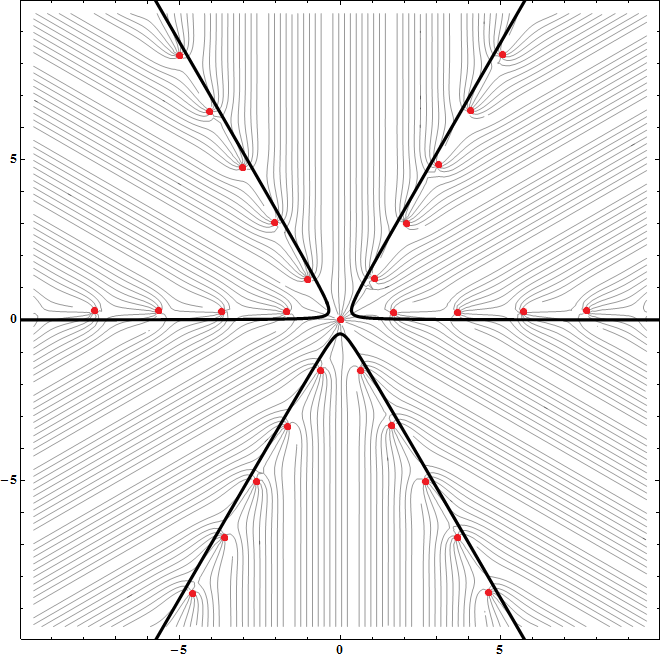}
        \caption{$\alpha = 0.5$}
    \end{subfigure}%
    ~
    \begin{subfigure}[t]{0.5\textwidth}
        \centering
        \includegraphics[height=0.7\textwidth]{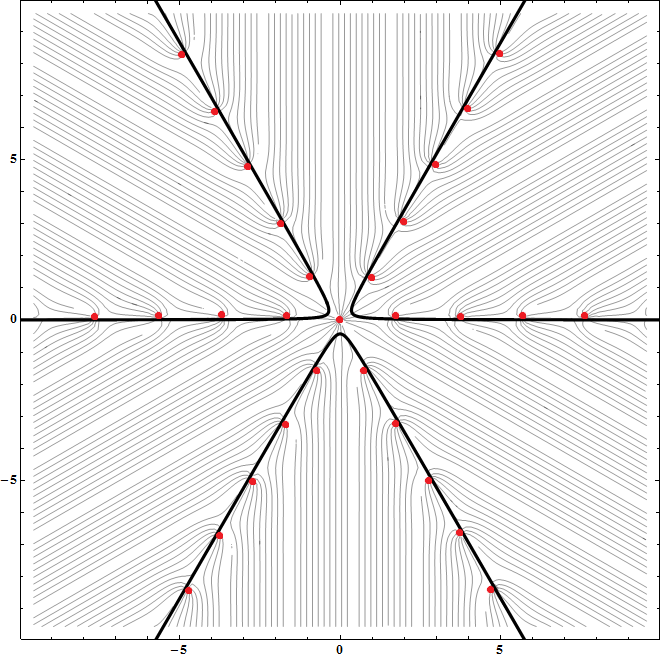}
        \caption{$\alpha = 0.8$}
    \end{subfigure}
\end{figure}

We will proceed as in the previous subsection, validating these deformation mappings for $0<\alpha<1$ by adding a forcing function, computing the $q(x,t)$ and then investigating the behaviour of the solution for various initial and boundary conditions. Note also that the truncation value of $\theta_{max}$ for $0<\alpha<1$ is the same as the $\alpha=0$ case, since the deformations functions used in both cases are the same.

\begin{example}
  \label{ex:5}
  Here we examine $\alpha = 1/3$. We select initial, boundary and forcing conditions such that
  \begin{eqnarray*}
    q(x,0)&=&0 \\
    q(0,t)&=&0 \\
    q(1,t)&=&2 \, \sin{(2\pi t)}\\
    q_{x}(1,t)&=&\frac{1}{3}q_{x}(0,t)\\
    h(x,t)&=&2\pi\,(3x-x^2)\cos{(2\pi t)}
  \end{eqnarray*}
  then the analytic solution of non-homogeneous Airy's equation is given by
  \begin{equation*}
    q(x,t) = (3x-x^2) \,\sin{(2\pi t)}.
  \end{equation*}
\end{example}
We substitute these initial, boundary and forcing terms into (\ref{eq:integral-rep-inc-forcing}) and make use of the contour deformations specified in (\ref{zerodeform}). We vary the truncation of the integral and illustrate the behaviour of the approximation in figure \ref{ex1alp05}.

\clearpage

\begin{figure}[h!]
    \centering
    \caption{\label{ex1alp05} A numerical benchmark of Airy's equation posed on a finite interval (\ref{lkdv2}). Initial, boundary and forcing functions are given in example \ref{ex:5}. Here $\alpha = 1/3$ and $\eta=1/2$. We test the effect of varying the truncation of the contour integrals in (\ref{eq:integral-rep-inc-forcing}) with $(x,t)\in[0,1]\times[0,1]$. Notice that as the contour length increases the error decreases quickly, as expected.}
    \begin{subfigure}[t]{0.5\textwidth}
        \centering
        \includegraphics[height=0.7\textwidth]{./fig/regions_with_names_and_contours_alpha0.pdf}
        \caption{The computational region and the deformation mappings.}
    \end{subfigure}%
    ~
    \begin{subfigure}[t]{0.5\textwidth}
        \centering
        \includegraphics[height=0.7\textwidth]{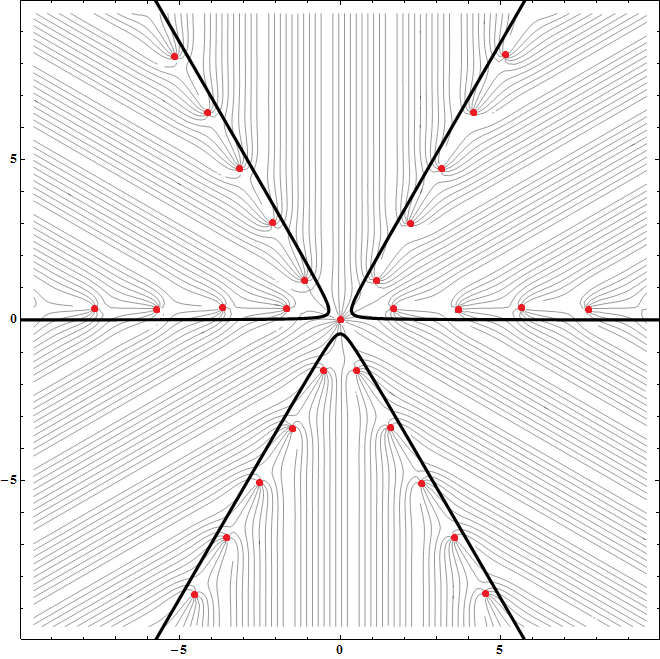}
        \caption{The location of the poles with respect to the deformation mappings.}
    \end{subfigure}
    \\
    \centering
    \begin{subfigure}[t]{0.5\textwidth}
        \centering
        \includegraphics[height=0.7\textwidth]{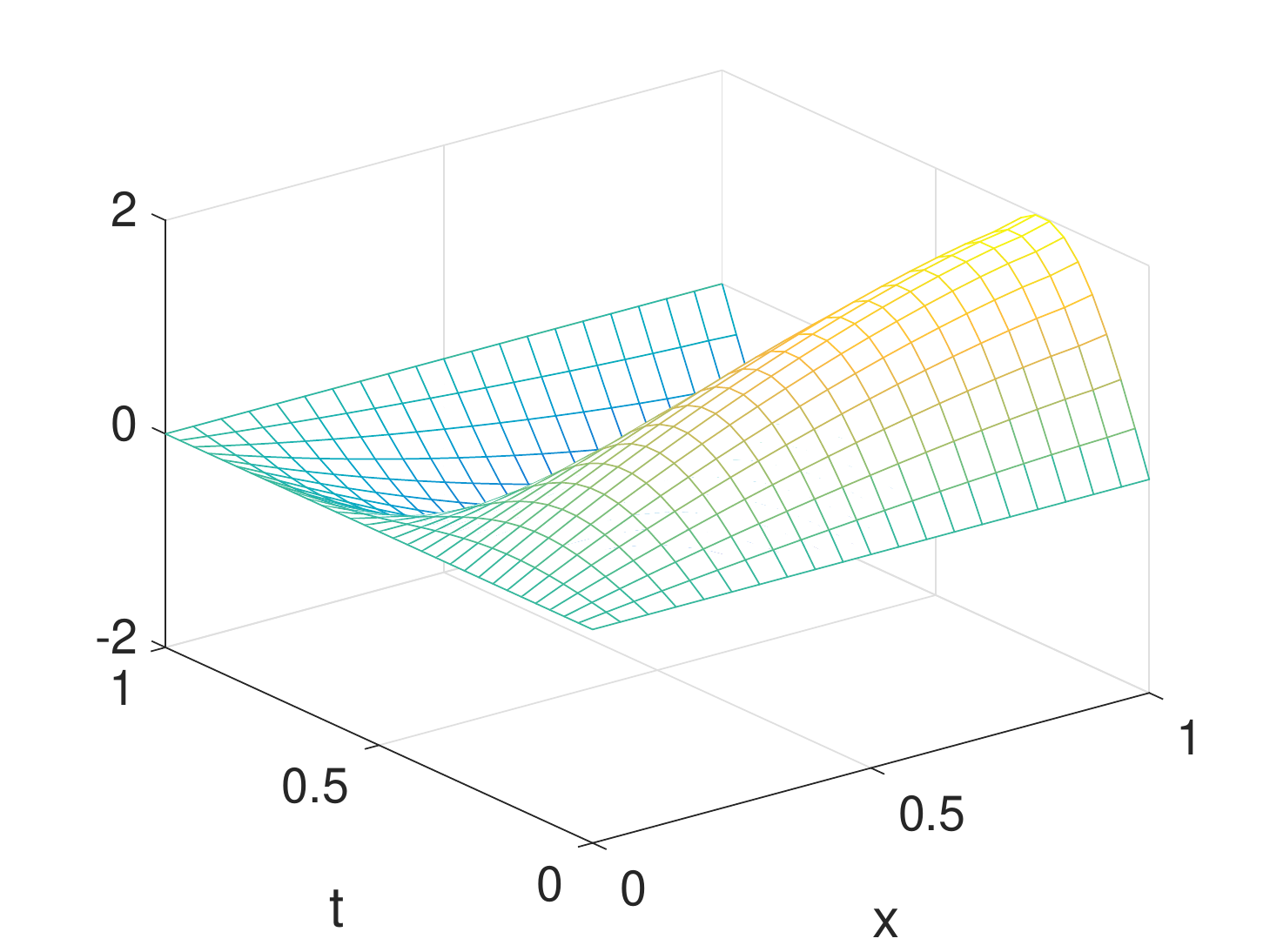}
        \caption{The numerical approximation with $\theta_{max} = 8$}
    \end{subfigure}%
    ~
    \begin{subfigure}[t]{0.5\textwidth}
        \centering
        \includegraphics[height=0.7\textwidth]{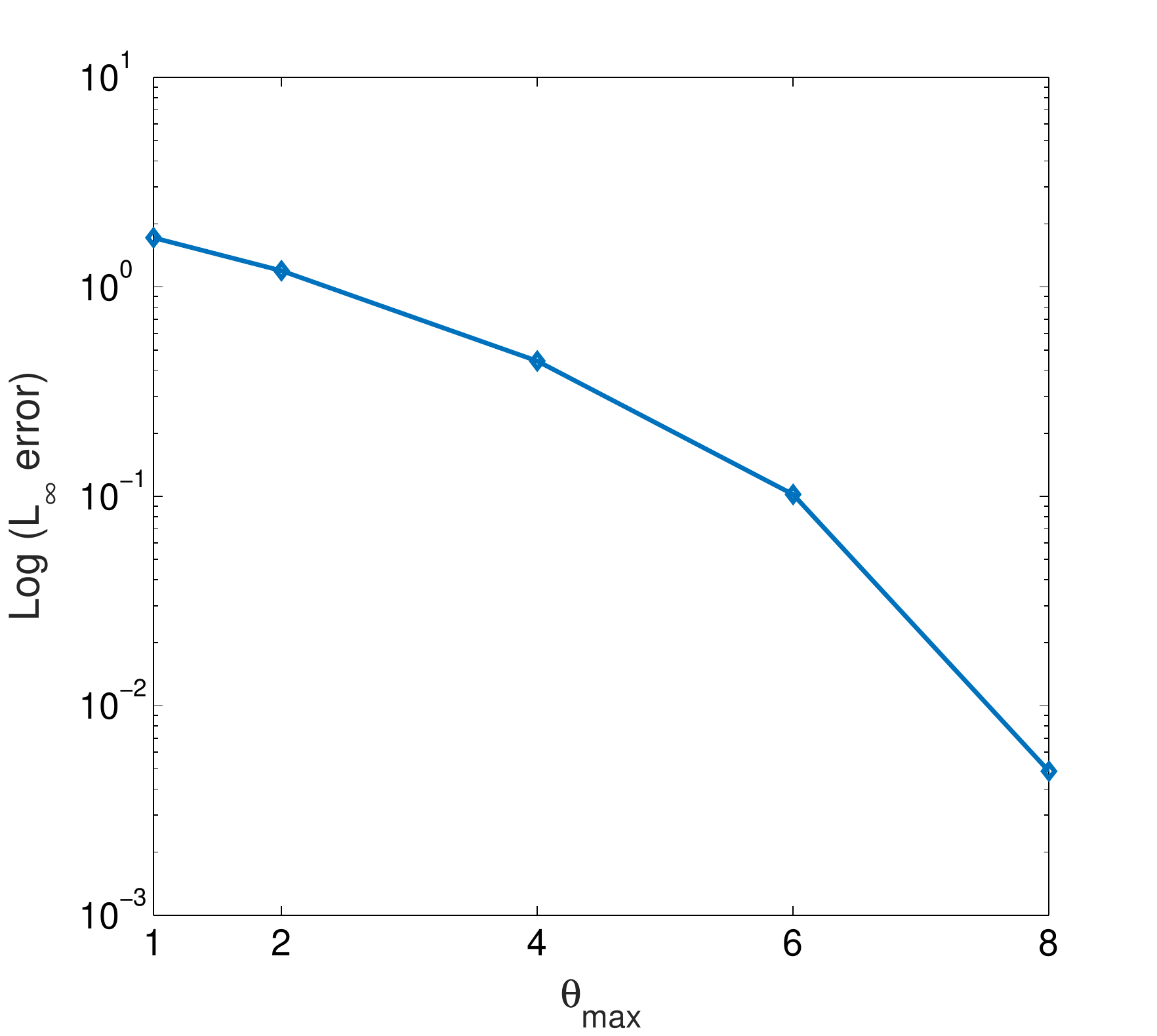}
        \caption{A plot of the maximum error in approximation against the contour truncation value.}
    \end{subfigure}
\end{figure}
\begin{example}
  In this example we show the approximate solution for two different choices of initial and boundary conditions. Note in both cases there is \emph{no forcing term}. In the first case we select initial and boundary conditions such that
  \begin{equation}
    \label{eq:alpha13eq1}
    \begin{split}
      q(x,0)&=1\\
      q(0,t)&=\cos{(2 \pi t)}\\
      q(1,t)&=\cos{(2 \pi t)}\\
      q_{x}(1,t)&= \frac 12 \, q_{x}(0,t).
    \end{split}
  \end{equation}
  The numerical approximation computed by the deformations defined for $\alpha=1/2$ above is given in figure \ref{fig:alpha13fig1}.

  In the second case we select initial and boundary conditions such that
  \begin{equation}
    \label{eq:alpha13eq2}
    \begin{split}
      q(x,0)&=e^{-12x}\\
      q(0,t)&=\cos{(2 \pi t)}\\
      q(1,t)&=\sin{(2 \pi t)}\\
      q_{x}(1,t)&=\frac{1}{e} \, q_{x}(0,t).
    \end{split}
  \end{equation}
  The numerical approximation computed by the deformations defined for $\alpha=1/e$ above is given in figure \ref{fig:alpha13fig2}.

  \begin{figure}[h!]
    \caption{\label{ex2sol} Examples of numerical approximations to Airy's equation using the Fokas transform method. }
    \centering
    \begin{subfigure}[t]{0.5\textwidth}
      \centering
      \includegraphics[height=0.7\textwidth,keepaspectratio=true]{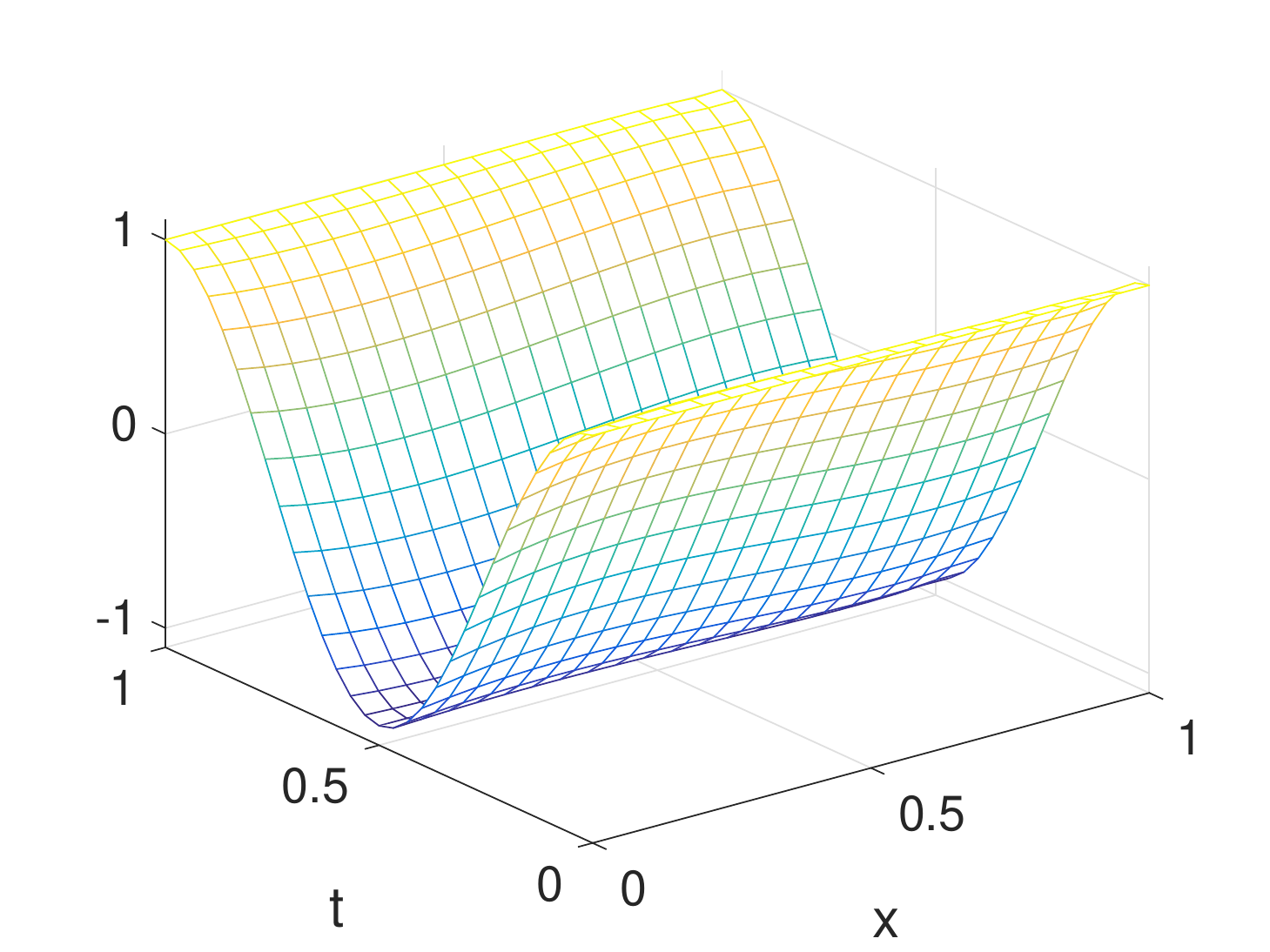}
      \caption{\label{fig:alpha13fig1} The solution given by the initial conditions (\ref{eq:alpha13eq1}).}
    \end{subfigure}%
    \begin{subfigure}[t]{0.5\textwidth}
      \centering
      \includegraphics[height=0.7\textwidth,keepaspectratio=true]{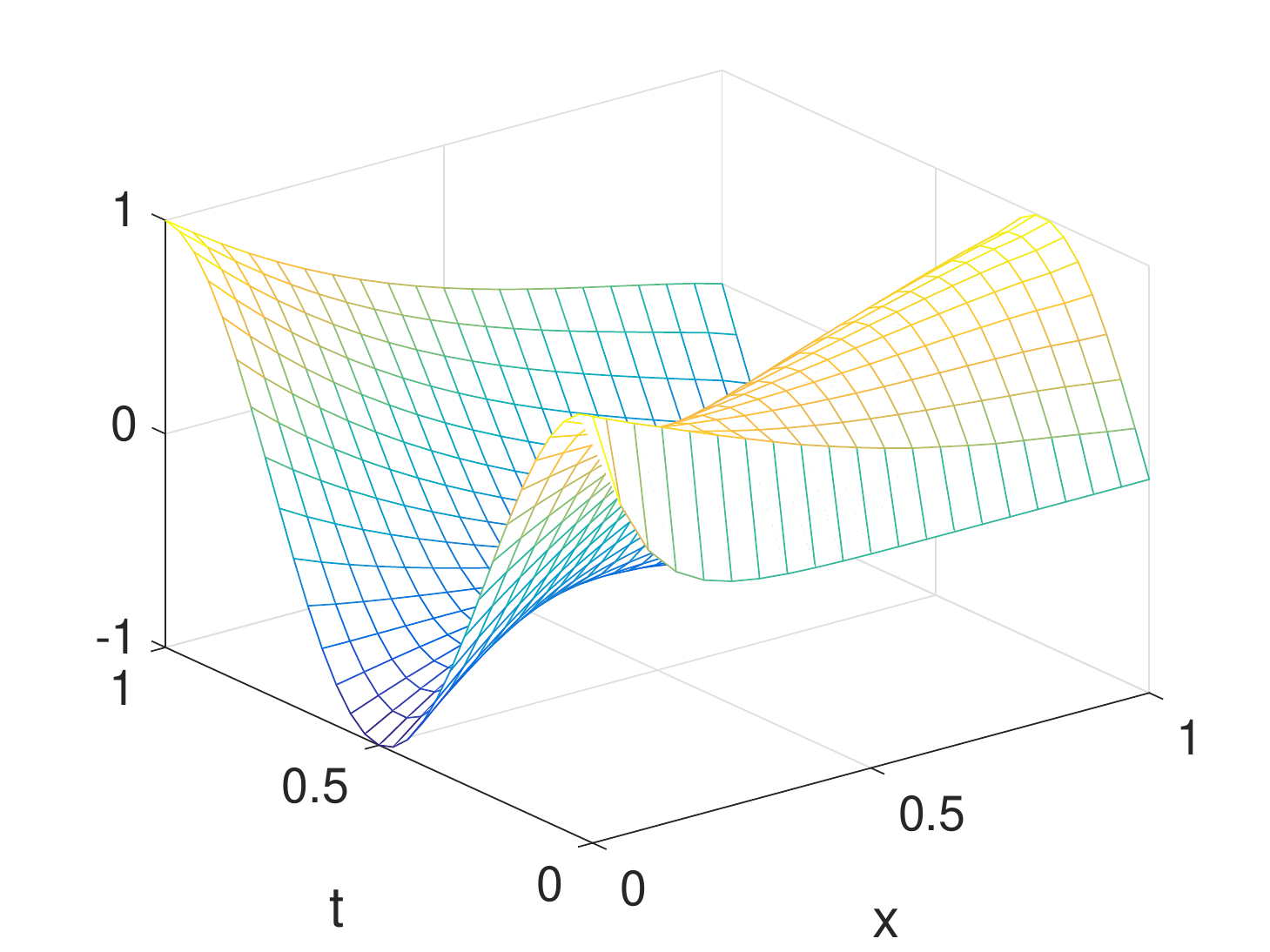}
      \caption{\label{fig:alpha13fig2} The solution given by the initial conditions (\ref{eq:alpha13eq2}).}
    \end{subfigure}%
  \end{figure}

\end{example}

\subsection{When $\alpha=1$}

  \begin{figure}[ht]
  \caption{
    \label{fig:pole-location}
    The location of the poles for the solution to Airy's equation when $\alpha=1$. We also show the behaviour of the deformation functions $k_1,k_2,k_3$ for some different values of $\gamma$ and $\beta$. The solid black line is the graph of the deformation function given by equation (\ref{alpha1deform})}
    \centering
    \begin{subfigure}[t]{0.5\textwidth}
        \centering
        \includegraphics[height=0.7\textwidth]{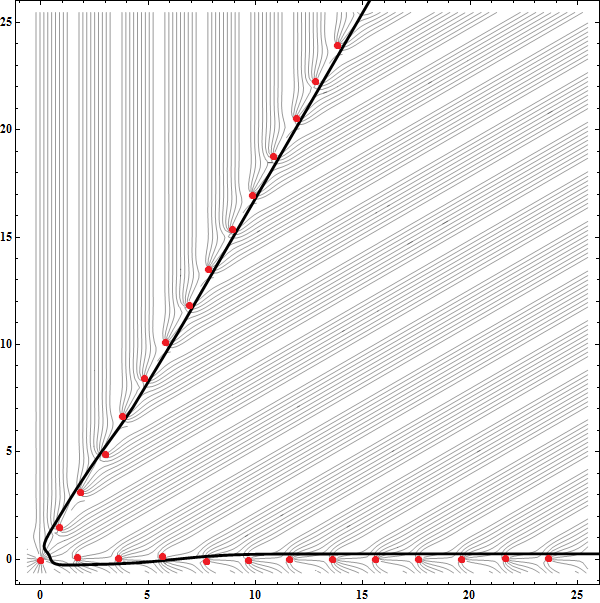}
        \caption{$\beta = 3$.}
    \end{subfigure}%
    \hfill
    \begin{subfigure}[t]{0.5\textwidth}
        \centering
        \includegraphics[height=0.7\textwidth]{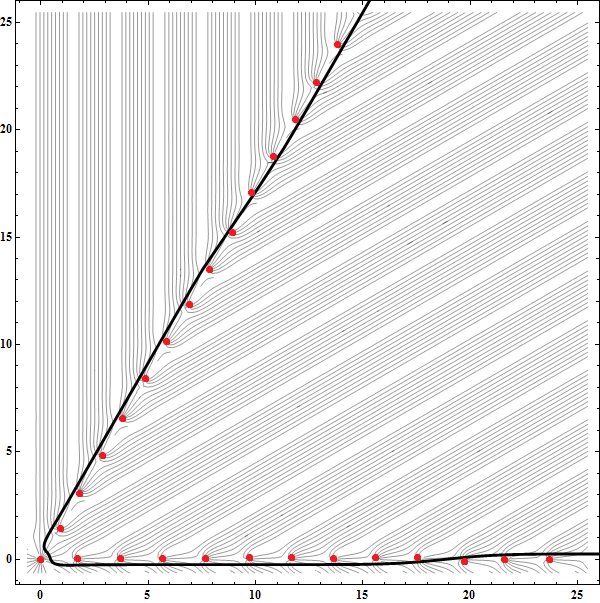}
        \caption{$\beta = 9$.}
    \end{subfigure}
    \\
    \begin{subfigure}[t]{0.47\textwidth}
        \centering
        \includegraphics[height=0.7\textwidth]{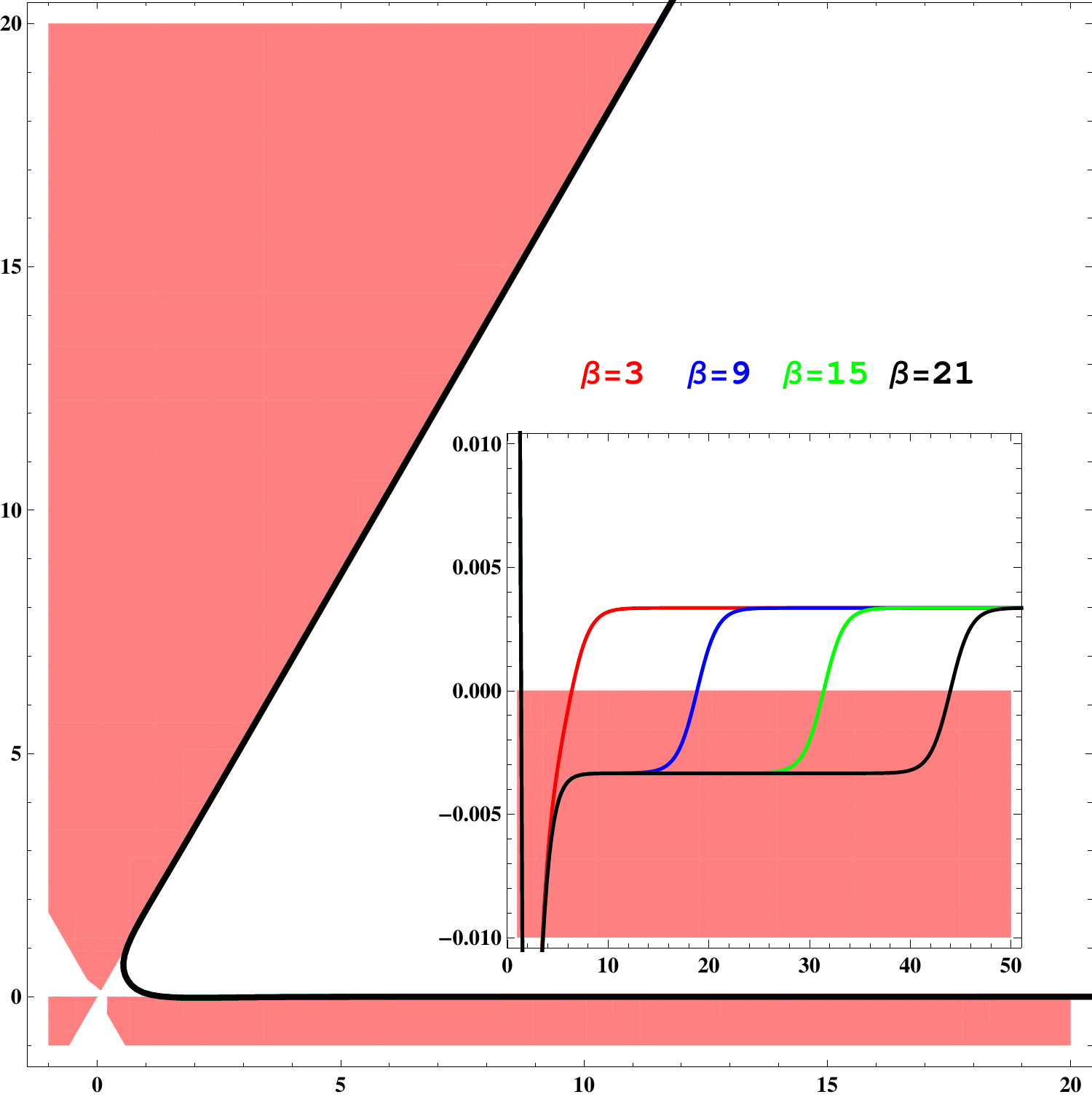}
        \caption{The deformation $k_2(\theta)$ for various values of $\beta$.}
    \end{subfigure}
    \hfill
    \begin{subfigure}[t]{0.47\textwidth}
        \centering
        \includegraphics[height=0.7\textwidth]{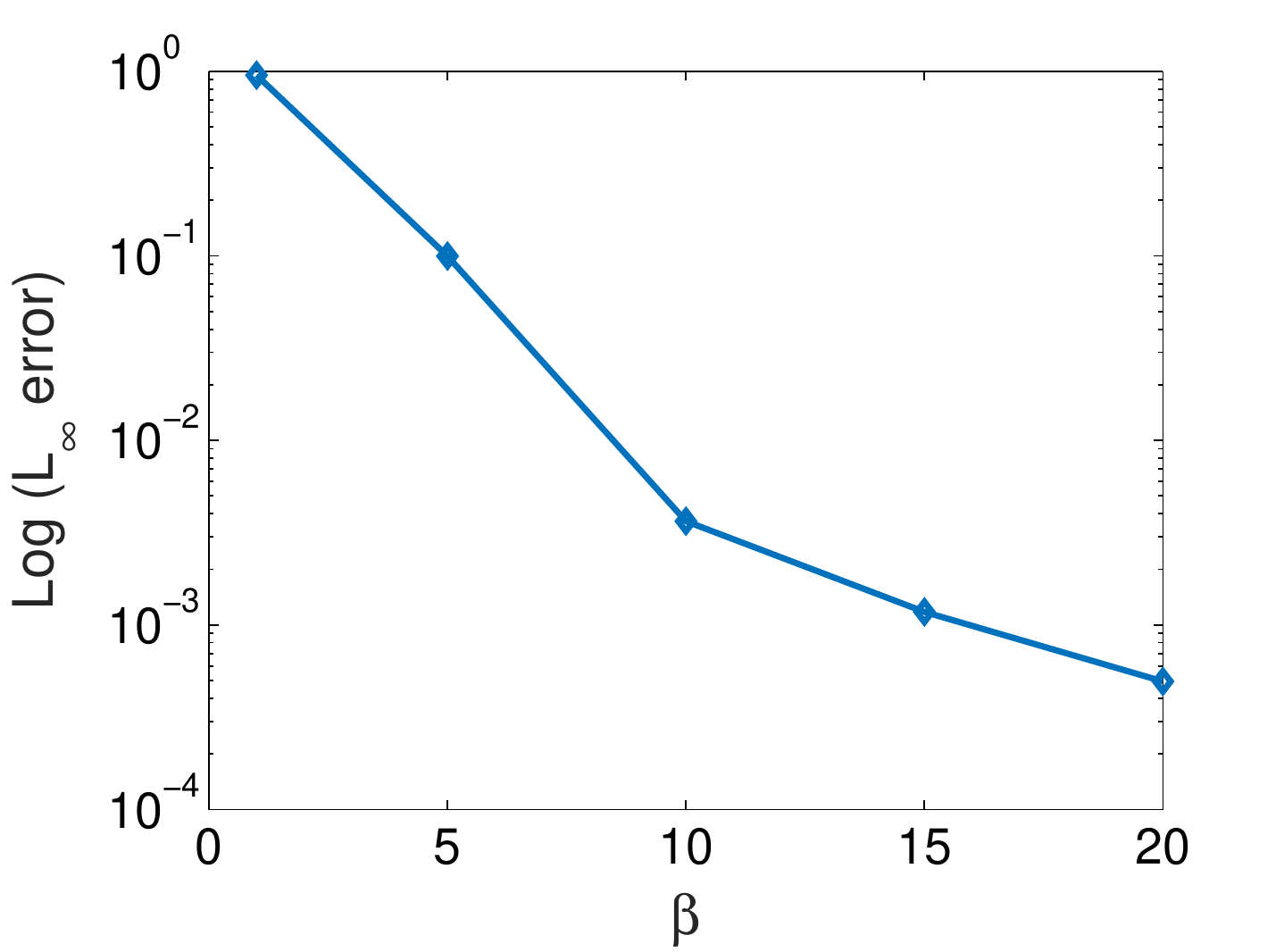}
        \caption{A plot of the maximum error in approximation against the number of poles added to the computation with $\theta_{max}=50$ and $(x,t)\in [0,1]\times[0,1]$.}
    \end{subfigure}%
\end{figure}

As can be seen from figure \ref{fig:pole-location} the problem is extremely challenging when $\alpha=1$. The poles lie on the boundaries of the regions we wish to integrate. We make use of the following deformation function:
\begin{eqnarray}\label{alpha1deform}
\begin{split}
k_{3}(\theta)=-\theta + & \gamma i \qc{\tanh{\qp{\beta \frac{\pi}{3}-\theta}}+\tanh{\qp{\beta \frac{\pi}{3}+\theta}} - 1} - i \qc{\tanh{\qp{\frac{1}{2}-2\theta}}+\tanh{\qp{\frac{1}{2}+2\theta}}} \\
               & + \frac{\sqrt{3}}{2} i \theta \tanh{\qp{2 \theta}}\qc{\tanh{\qp{\frac{1}{2}-2\theta}}+\tanh{\qp{\frac{1}{2}+2\theta}} - 2},
\end{split}
\end{eqnarray}
\begin{equation*}
k_{1}(\theta)= \tau^2 \, k_{3}(\theta), \quad k_{2}(\theta)= \tau \, k_{3}(\theta)
\end{equation*}
{where $\gamma$ and $\beta$ are the positive real constants. The parameter $\gamma$ determines the distance between the curve and the poles. We will set $\gamma=\min{\{\frac{1}{200 \sqrt[3]{t}}, 0.5\}}$ to guarantee minimum passage through $D^{\pm}$ regions in which $e^{ik^3 t}$ is unbounded for larger time values. The parameter $\beta$ determines the number of poles added to the numerical computation by shifting the curve from $D^{\pm}$ regions to the $E^{\pm}$ regions. As mentioned before, the poles are placed on the $\tau^j\mathbb{R}$ lines for $j=0,1,2$ when $\alpha=1$.

  \clearpage
\begin{example}
  Here we examine the case $\alpha = 1$. We select initial, boundary and forcing conditions such that
  \begin{equation}
    \begin{split}
      q(x,0)&=\sin{(2\pi x)}\\
      q(0,t)&=0 \\
      q(1,t)&=0 \\
      q_{x}(1,t)&=q_{x}(0,t)\\
      h(x,t)&=-(2\pi)^3\cos{(2\pi x)},
    \end{split}
  \end{equation}
  then the analytic solution of the non-homogeneous Airy's equation is given by
  \begin{equation*}
    q(x,t) = \sin{(2\pi x)}.
  \end{equation*}
We implement these initial and boundary values into (\ref{eq:integral-rep-inc-forcing}) and make use of the contour deformations specified in (\ref{alpha1deform}). We vary the truncation of the integral and illustrate the behaviour of the approximation in figure (\ref{fig:alpha1error}). We get the numerical solution given in figure (\ref{fig:alpha1soln}). The truncation value of $\theta_{max}$ in $\alpha=1$ case is larger than $\alpha=0$ and $0<\alpha<1$ cases. In this case, after including contributions from the poles located on the $\tau^j\mathbb{R}$ lines for $j=0,1,2$, the contour stays inside the $E^{\pm}$ regions in which the integrands of the integral representation are bounded.

\begin{figure}[h!]
  \centering
  \caption{\label{fig:alpha1} A numerical benchmark of Airy's equation posed on a finite interval (\ref{lkdv2}). Initial, boundary and forcing functions are given in example \ref{ex:5}. Here $\alpha = 1$. We test the effect of varying the truncation of the contour integrals in (\ref{alpha1deform}) with $\gamma=\min{\{\frac{1}{200 \sqrt[3]{t}}, 0.5\}}$, $\beta=9$ and $(x,t) \in [0,1]\times[0,10]$. Notice that as the contour length increases the error decreases quickly, as expected.
}
    \begin{subfigure}[t]{0.47\textwidth}
        \centering
        \includegraphics[height=0.7\textwidth]{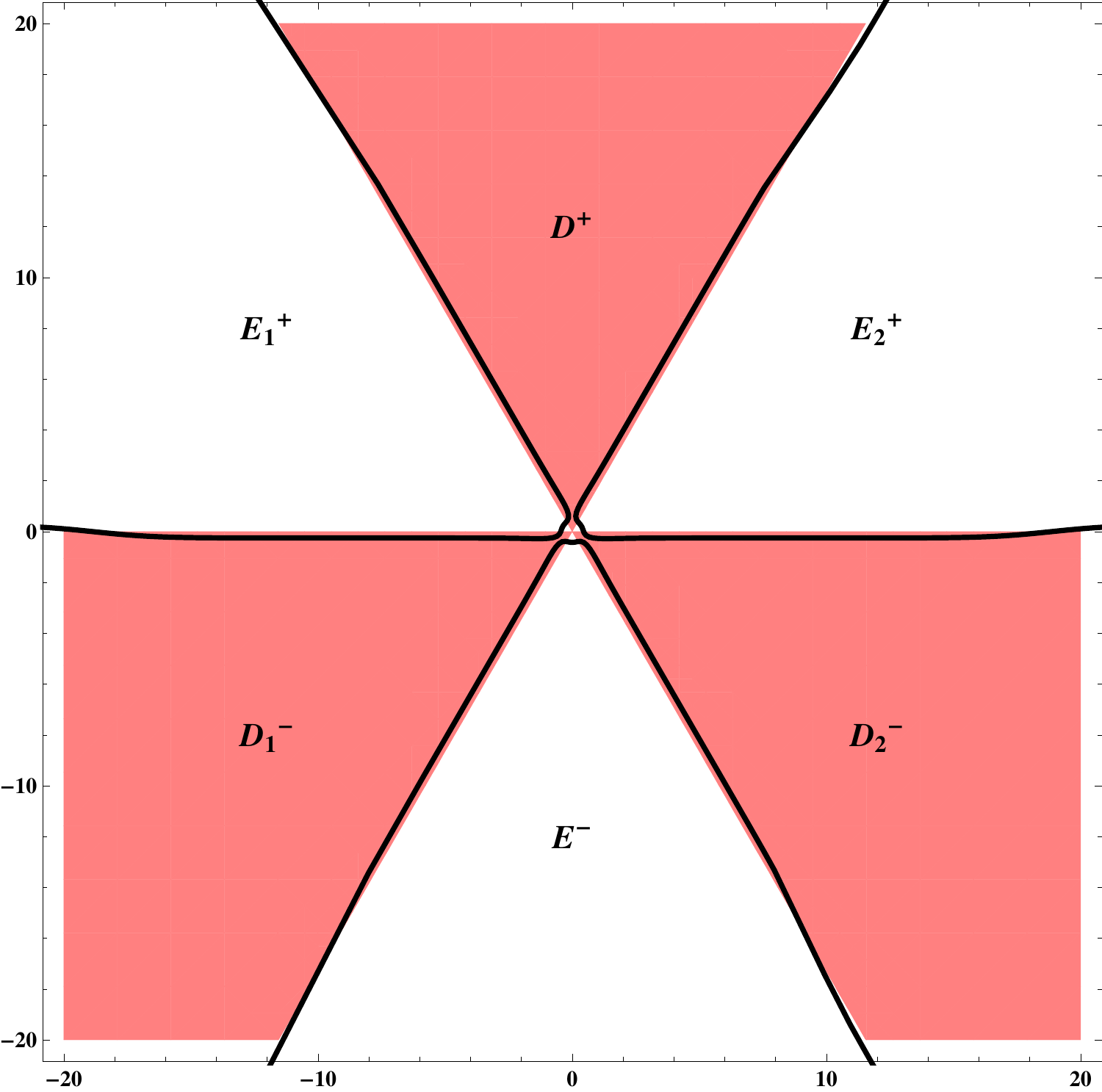}
        \caption{The computational region and deformation mappings.}
    \end{subfigure}%
    \hfill
    \begin{subfigure}[t]{0.47\textwidth}
        \centering
        \includegraphics[height=0.7\textwidth]{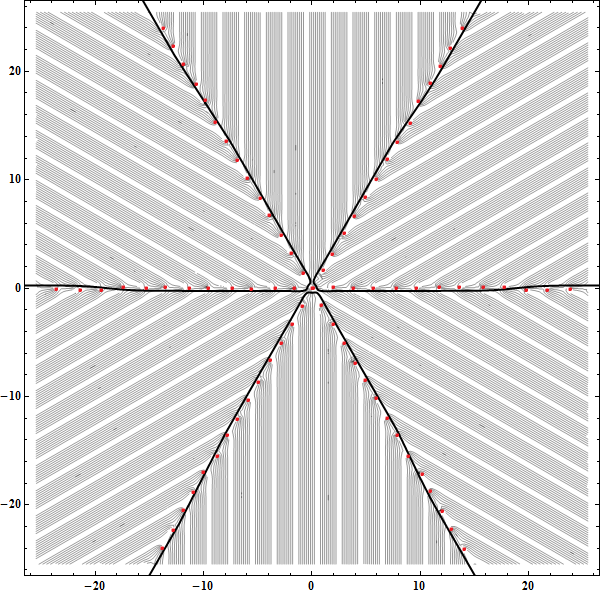}
        \caption{The location of the poles with respect to the deformation mappings.}
    \end{subfigure}%
    \\
    \begin{subfigure}[t]{0.47\textwidth}
        \centering
        \includegraphics[height=0.7\textwidth]{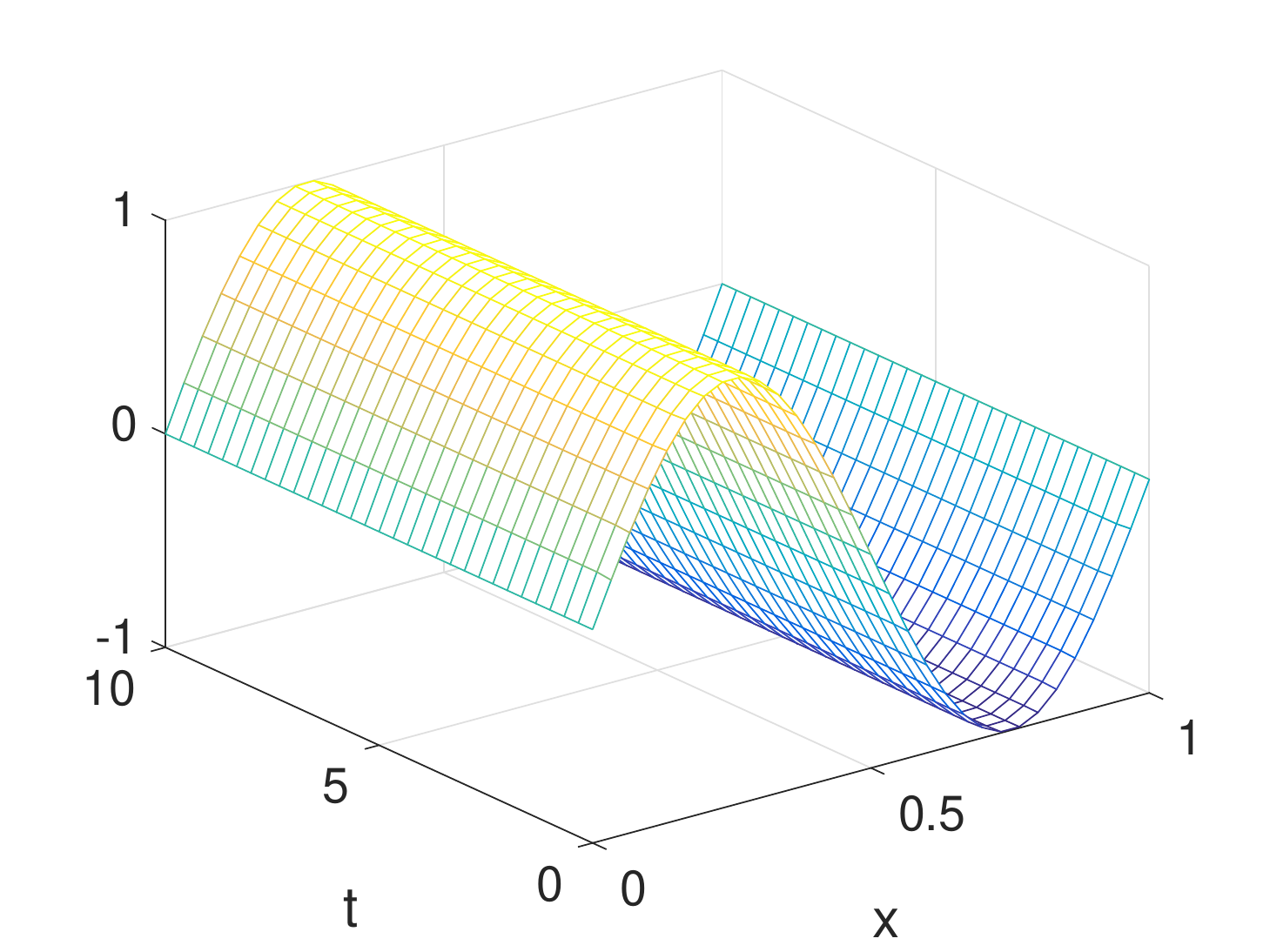}
        \caption{\label{fig:alpha1soln}The numerical approximation with $\theta_{max} = 50$.}
    \end{subfigure}%
    \hfill
    \begin{subfigure}[t]{0.47\textwidth}
        \centering
        \includegraphics[height=0.7\textwidth]{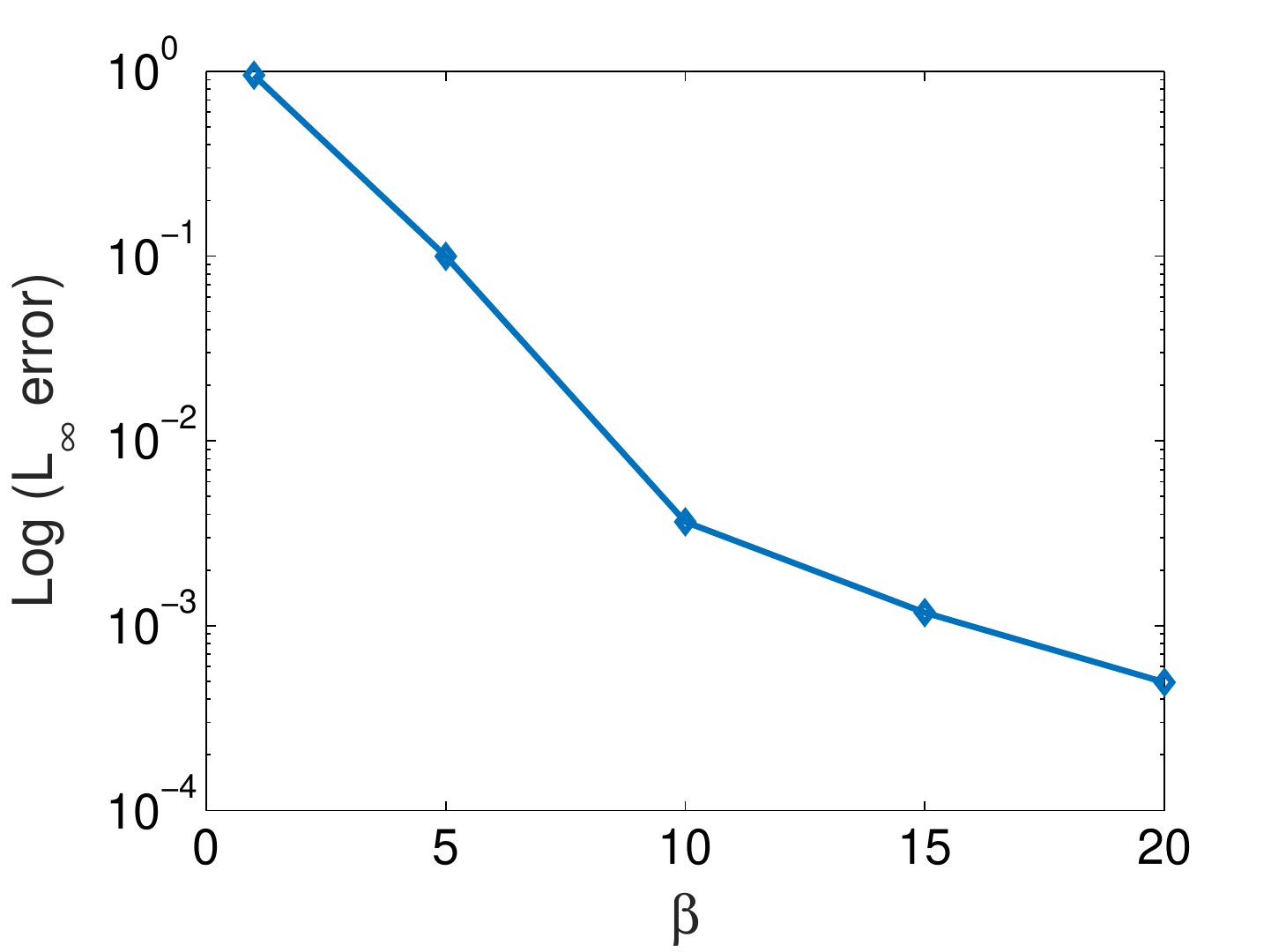}
        \caption{\label{fig:alpha1error}A plot of the maximum error in approximation against the contour truncation value.}
    \end{subfigure}
\end{figure}
\end{example}
\begin{example}
  \label{ex:7}
  In this example we show the approximate solution for a specific choice of initial and boundary conditions. We will examine the effect of varying $\alpha$ on the solution. Note in all cases there is \emph{no forcing term}. We select initial and boundary conditions such that
  \begin{equation}
    \label{eq:alpha1eq1}
    \begin{split}
      q(x,0)&=\sin(2 \pi x)\\
      q(0,t)&=0\\
      q(1,t)&=0\\
      q_{x}(1,t)&=\alpha q_{x}(0,t).
    \end{split}
  \end{equation}
  The numerical approximation computed by the deformations defined for $\alpha=1$ above is given in figure \ref{fig:alpha1fig1}. A slice of the solution is taken at time $t=0.1$ for various values of $\alpha$ and is shown in figure \ref{ex2sol}.

  \begin{figure}[h!]
    \caption{\label{ex2sol} Examples of numerical approximations to Airy's equation using the Fokas transform method. The parameters of the deformations are chosen as $\gamma=\min{\{\frac{1}{200 \sqrt[3]{t}}, 0.5\}}$, $\beta=9$ and $\theta_{max}=50$.}
    \centering
    \begin{subfigure}[t]{0.5\textwidth}
      \centering
      \includegraphics[height=0.7\textwidth,keepaspectratio=true]{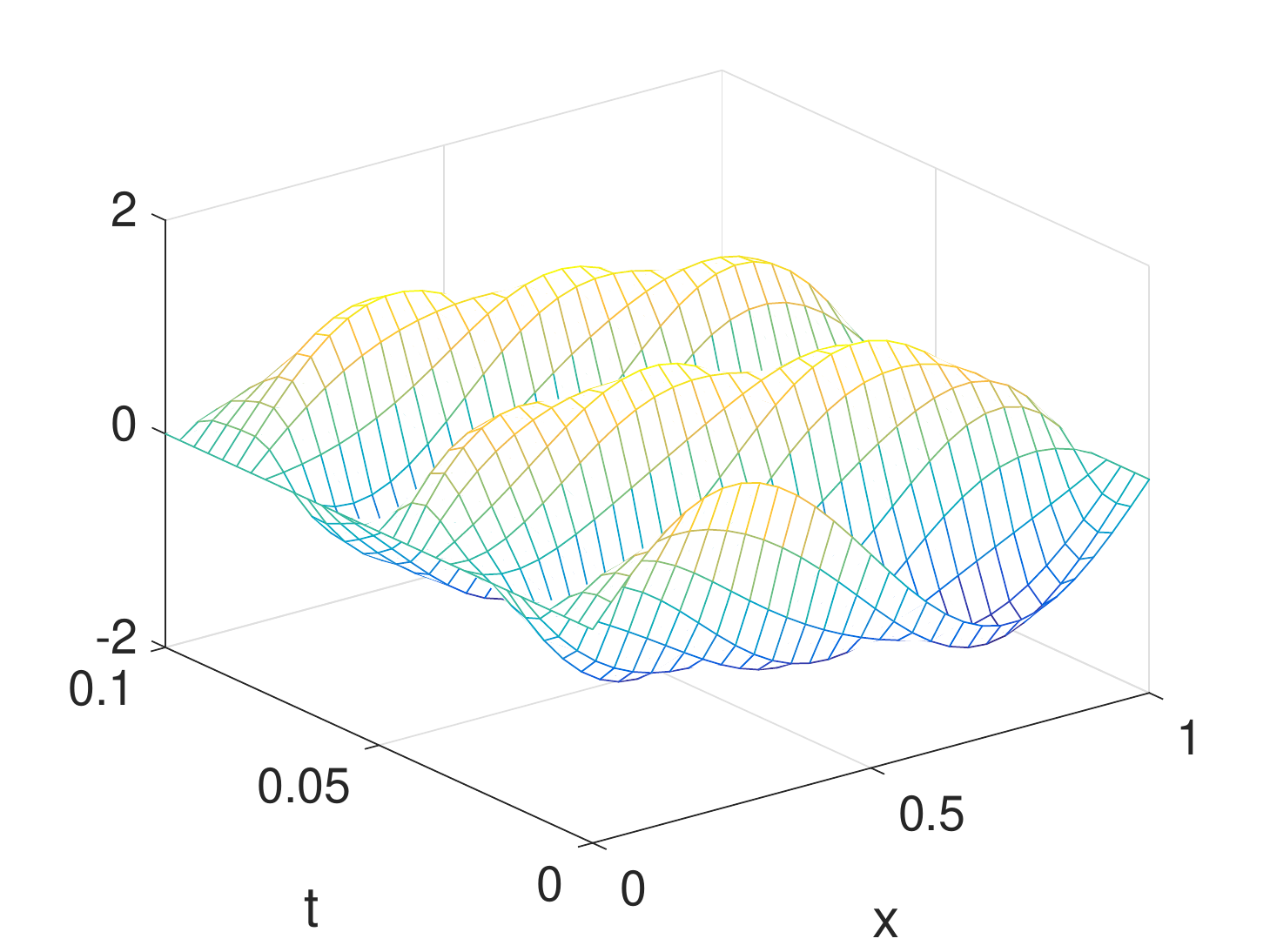}
      \caption{\label{fig:alpha1fig1} The solution given by the initial conditions (\ref{eq:alpha1eq1}) for $\alpha=1$.}
    \end{subfigure}%
    \hfill
    \begin{subfigure}[t]{0.5\textwidth}
      \centering
      \includegraphics[height=0.7\textwidth,keepaspectratio=true]{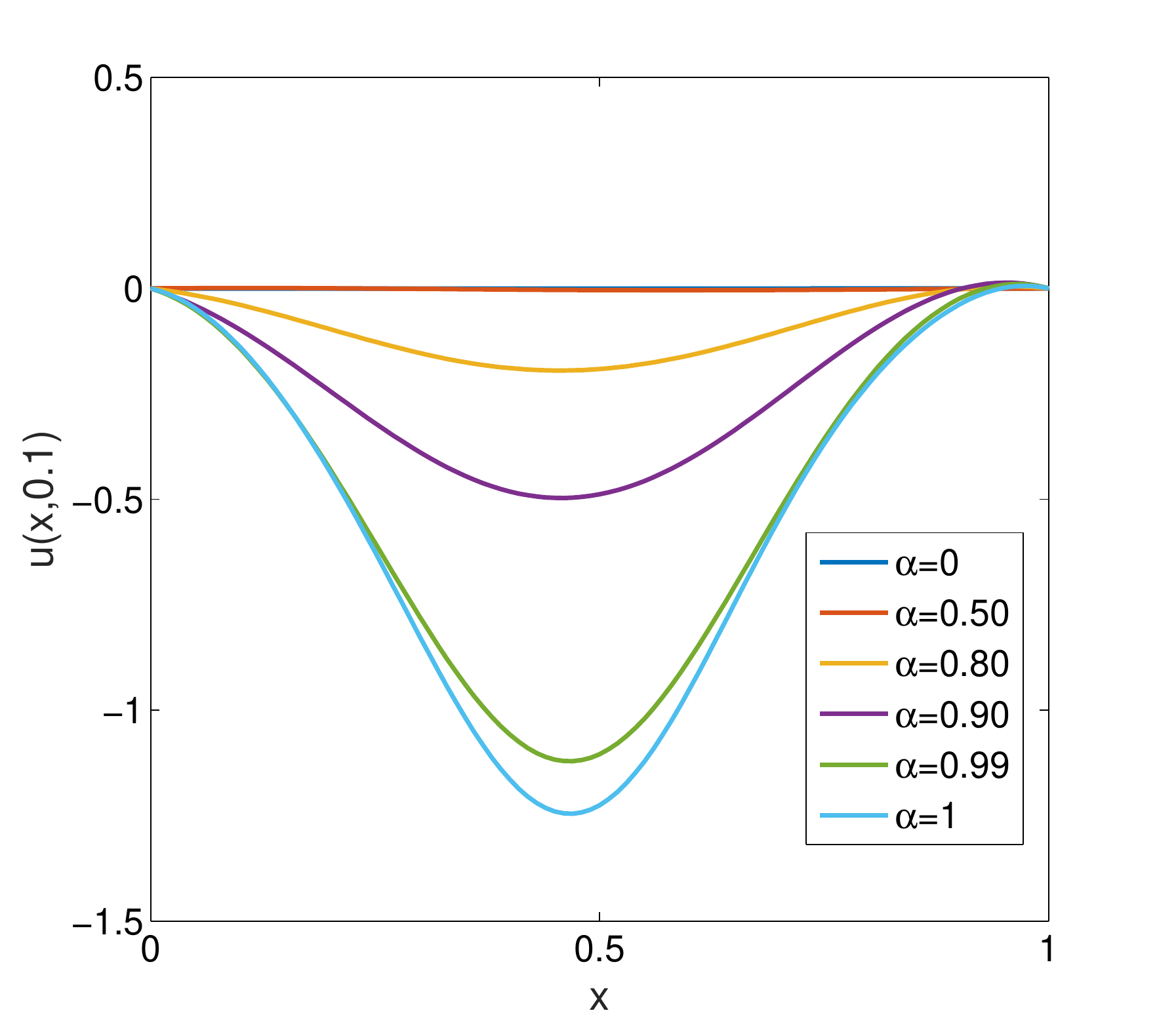}
      \caption{\label{fig:alpha1fig2} The solution given by the initial conditions (\ref{eq:alpha1eq1}) sliced at $t=0.1$ for various values of $\alpha$.}
    \end{subfigure}%
  \end{figure}
\end{example}

\section{Conclusions}
We have conducted a careful numerical evaluation of the solution of a variety of boundary value problems for the third-order PDE (\ref{lkdv}), and shown how this evaluation is extremely sensitive to the choice of integration contour. We also demonstrated that the optimal strategy for choosing the contour depends on the particular boundary conditions, namely it depends on the location of the poles of the integrand,  and on how many of these poles are included - which is equivalent to the residue at the pole being included in the final result of the computation. Note that these poles correspond to the eigenvalues of the spatial linear differential operator defined by the PDE and the boundary conditions.

Owing to the fast exponential decay of the integrand,  in every case only a few poles need be included. This should be contrasted with the slow convergence of series representations of the solutions of such problems when such a representation exists. The present numerical strategy is not only the only general one in existence, to our knowledge, for third-order 2-point boundary value problems, but also a competitive one for evaluating the solution of more classical 2-point boundary value problems for well known equations such as the heat equation. We have not given a rigorous proof that our choice is optimal, and we do not claim it to be. The rigorous study of the optimal strategy for choosing the deformation is a fundamental and interesting mathematical question, but it  is outside the scope of the present study.

The ability to conduct such numerical computations enables the study of the behaviour of the solution of such problems in the case that the initial condition is a piecewise constant function. This is the case studied by Olver in the case of periodic boundary conditions, and displaying the phenomenon of {\em dispersive quantisation}.  For general boundary conditions, the eigenvalues cannot be computed explicitly as they are the roots of a transcendental equation. Therefore it is not possible to verify analytically by an analogous computation to the one performed by Olver in his paper whether dispersive quantisation is expected to occur. A numerical study of the phenomenon is made possible by the results presented in this paper, and will be undertaken in subsequent work.


\vspace*{30pt}

\bibliographystyle{elsarticle-num}
\bibliography{References}

\end{document}